\documentclass[11pt,english]{article}
\usepackage[margin= 2 cm,bottom=23mm, top= 20mm]{geometry}

\usepackage{amsthm}
\usepackage{amsmath}
\usepackage{amssymb}
\usepackage{setspace}
\usepackage{mathtools}
\usepackage{verbatim}
\usepackage{csquotes}
\usepackage[hidelinks]{hyperref}
\usepackage{thm-restate}
\usepackage{cleveref}
\usepackage[shortlabels]{enumitem}
\setlist[itemize]{leftmargin=*} %
\setlist[enumerate]{leftmargin=*}
\usepackage{framed}
\usepackage{floatrow}
\usepackage[T1]{fontenc}
\usepackage{bbm}
\usepackage[color=Orange!50!white, textwidth=22mm]{todonotes}

\usepackage{comment}

\usepackage{mathdots}
\usepackage{xcolor}
\usepackage{diagbox}
\usepackage{colortbl}

\usepackage{graphicx}
\usepackage{subcaption}

\theoremstyle{plain}

\newtheorem{theorem}{Theorem}[section]
\newtheorem{claim}[theorem]{Claim}

\newtheorem{lemma}[theorem]{Lemma}
\newtheorem{corollary}[theorem]{Corollary}
\newtheorem{conjecture}[theorem]{Conjecture}

\newtheorem{problem}[theorem]{Problem}
\theoremstyle{definition}

\newtheorem*{defn*}{Definition}

\expandafter\def\expandafter\normalsize\expandafter{%
    \normalsize
    \setlength\abovedisplayskip{4pt}
    \setlength\belowdisplayskip{4pt}
    \setlength\abovedisplayshortskip{4pt}
    \setlength\belowdisplayshortskip{4pt}
}

\usepackage[square,sort,comma,numbers]{natbib}
\newcommand{\calP}{\mathcal{P}}

\def\eps {\varepsilon}

\newcommand{\cP}{\mathcal{P}}

\newcommand{\ex}{\mathrm{ex}}

\renewcommand{\Pr}{\mathbb{P}}

\title{The largest subgraph without a forbidden induced subgraph}
\author{Jacob Fox\thanks{Department of Mathematics, Stanford University, Stanford, CA. Email: \textbf{jacobfox@stanford.edu}. Research
supported by NSF Award DMS-2154129.} 
\and Rajko Nenadov\thanks{School of Computer Science, University of Auckland, New Zealand. Email: \textbf{rajko.nenadov@auckland.ac.nz}. Research supported by the New Zealand Marsden Fund.}
\and Huy Tuan Pham\thanks{Department of Mathematics, Stanford University, Stanford, CA. Email: \textbf{huypham@stanford.edu}. Research supported by a Clay Research Fellowship and a Stanford Science Fellowship.}
\footnotemark[1]}
\date{}

\begin{document}
    
\maketitle

\begin{abstract}
We initiate the systematic study of the following Tur\'an-type question. Suppose $\Gamma$ is a graph with $n$ vertices such that the edge density between any pair of subsets of vertices of size at least $t$ is at most $1 - c$, for some $t$ and $c > 0$. What is the largest number of edges in a subgraph $G \subseteq \Gamma$ which does not contain a fixed graph $H$ as an induced subgraph or, more generally, which belongs to a hereditary property $\mathcal{P}$? This provides a common generalization of two recently studied cases, namely $\Gamma$ being a (pseudo-)random graph and a graph without a large complete bipartite subgraph. We focus on the interesting case where $H$ is a bipartite graph.

We determine the answer up to a constant factor with respect to $n$ and $t$, for certain bipartite $H$ and for $\Gamma$ either a dense random graph or a Paley graph with a square number of vertices. In particular, our bounds match if $H$ is a tree, or if one part of $H$ has $d$ vertices complete to the other part, all other vertices in that part have degree at most $d$, and the other part has sufficiently many vertices. As  applications of the latter result, we answer a question of Alon, Krivelevich, and Samotij on the largest subgraph with a hereditary property which misses a bipartite graph, and determine up to a constant factor the largest number of edges in a string subgraph of $\Gamma$. The proofs are based on a variant of the dependent random choice and a novel approach for finding induced copies by inductively defining probability distributions supported on induced copies of smaller subgraphs.

\end{abstract}

\section{Introduction} 

The Tur\'an number $\mathrm{ex}(n,H)$ is the maximum number of edges in a graph on $n$ vertices which does not contain $H$ as a subgraph. The Erd\H{o}s-Stone-Simonovits theorem \cite{ES46,ES66} asymptotically determines Tur\'an numbers for fixed non-bipartite $H$. It states that $\mathrm{ex}(n,H)=(1-\frac{1}{\chi(H)-1}+o(1)){n \choose 2}$, where $\chi(H)$ is the chromatic number of $H$, and the $o(1)$ term goes to $0$ for $H$ fixed as $n \to \infty$. For $H$ bipartite, the Erd\H{o}s-Stone-Simonovits theorem gives $\mathrm{ex}(n,H)=o(n^2)$, and it is further known that there is $\varepsilon=\varepsilon(H)>0$ such that $\mathrm{ex}(n,H) \leq n^{2-\varepsilon}$. However, it remains a challenging open problem to give good estimates of Tur\'an numbers of bipartite graphs. A prominent case is complete bipartite graphs $K_{s,r}$ with $s \leq r$, for which the K\"ov\'ari-S\'os-Tur\'an theorem \cite{KST} gives $\mathrm{ex}(n,K_{s,r})=O_r(n^{2-1/s})$. This bound is known to be sharp up to the constant factor if $s=2$ or $s=3$, or $r$ is sufficiently large as a function of $s$ (see  \cite{ARS,KRS,Bukh}). 

A property of graphs is \emph{hereditary} if it is closed under taking induced subgraphs. A hereditary property of graphs is nontrivial if it contains all empty graphs and misses some graph. Throughout we assume all hereditary properties of graphs are nontrivial. For a hereditary property $\mathcal{P}$ of graphs and a graph $\Gamma$, let $\mathrm{ex}(\Gamma,\mathcal{P})$ be the maximum number of edges of a subgraph $G \subseteq \Gamma$ that belongs to $\mathcal{P}$. Note that in the  special case that $\Gamma$ is the complete graph $K_n$ and $\mathcal{P}$ is the property of not containing $H$ as a subgraph, the definition coincides with that of $\mathrm{ex}(n,H)$. 

Alon, Krivelevich, and Samotij \cite{AKS23} proved the following extension of the Erd\H{o}s-Stone-Simonovits theorem: For $0<p \leq 1$ fixed, if $k=k(\mathcal{P}) \geq 2$ is the minimum chromatic number of a graph that does not belong to $\mathcal{P}$, then the Erd\H{o}s-Renyi random graph $\Gamma = G(n, p)$ satisfies $\mathrm{ex}(\Gamma,\mathcal{P})=(1-\frac{1}{k-1}+o(1))p{n \choose 2}$ with high probability. While the proof they give does not extend from Erd\H{o}s-Renyi graphs to quasirandom graphs, a different proof based on Szemeredi's regularity lemma mentioned in the concluding remarks of their paper does extend to quasirandom graphs, but gives weak quantitative bounds. They left as an open problem to improve the bound $\mathrm{ex}(\Gamma,\mathcal{P})=o(n^2)$ in the case $k=2$ to 
$\mathrm{ex}(\Gamma,\mathcal{P}) \leq n^{2-\varepsilon}$ for some $\varepsilon=\varepsilon(\mathcal{P})>0$. We verify that this is the case. 

Our results hold for a much larger class of graphs $\Gamma$. We say that a graph $\Gamma$ is \emph{$(c,t)$-sparse} if for every pair of vertex subsets $A,B \subset V(G)$ (not necessarily disjoint) with $|A|,|B| \geq t$, we have $e(A,B) \le (1-c)|A||B|$. Here, $e(A, B)$ counts the number of pairs $(a, b) \in A \times B$ such that $\{a,b\}$ is an edge in $\Gamma$. In particular, any edge which lies in $A \cap B$ is counted twice. Roughly, a graph is $(c, t)$-sparse if it doesn't have large, almost complete bipartite subgraphs. If a graph is $(c,t)$-sparse, then all of its subgraphs are as well. 

Generalizing the question of Alon, Krivelevich, and Samotij \cite{AKS23}, we propose the following  Tur\'an-type problem. 
\begin{problem}\label{prop:main}
    Given a hereditary property $\mathcal{P}$ which misses a bipartite graph and a $(c, t)$-sparse graph $\Gamma$, determine $\mathrm{ex}(\Gamma, \mathcal{P})$.
\end{problem}
Note that, if $\Gamma=G(n,p)$ with $0<p<1$ fixed, with high probability $\Gamma$ is $(c,t)$-sparse for any fixed $c<1-p$ and any $t \geq C'\log n$ where $C'=C'(p,c)$ is sufficiently large. Problem \ref{prop:main} thus entails the study of Tur\'an-type problems in dense random (and pseudorandom) graphs. Problem \ref{prop:main} also entails the case $\Gamma$ is $K_{t,t}$-free (taking $c=1/t^2$), which has been the subject of recent progress \cite{AZ, BBCD, GM, HMST,MST1,MST2,Zim}. We refer the reader to Subsection \ref{subsec:no-complete-bip} for further details on the relation between our work and these results.  

The following extension of the K\"ov\'ari-S\'os-Tur\'an theorem is our first main result. We say a bipartite graph $H = (A \cup B, E)$ is \emph{$d$-bounded} if there are $1 \le \ell \le d$ vertices in $A$ which are complete to $B$, and all other vertices in $A$ have degree at most $d$. 

\begin{theorem}\label{thm:drc}
Let $H$ be a $d$-bounded bipartite graph and $c > 0$. Then there are constants $C$ and $\alpha > 0$ such that if $\Gamma$ is a $(c,t)$-sparse graph with $n$ vertices, for some $t$, then any subgraph $G \subseteq \Gamma$ with at least $Ct^{1/d}n^{2-1/d}$ edges contains $\alpha  q^{e(H)}n^{v(H)}$ copies of $H$ which are induced copies in $\Gamma$, where $q = 2e(G) / n^2$.
\end{theorem} 

The bound on the number of edges in \Cref{thm:drc} is best possible up to the constant factor $C$ in various cases as discussed in the next subsection, thus giving a complete answer (up to the constant factor) to Problem \ref{prop:main} in these cases. Also, the bound on the number of induced copies of $H$ in $G$ is best possible up to the constant factor $\alpha$ if for instance $G$ is a random graph with edge density $q$. 

If $H = (A \cup B, E)$ is such that all vertices in $A$ have degree at most $d$, then $H$ is a subgraph of a $d$-bounded graph. Thus we immediately get the following corollary, albeit without a tight bound on the number of copies.

\begin{corollary} \label{cor:deg_d}
    Let $H$ be a bipartite graph with all vertices in one part having degree at most $d$. If $\Gamma$ is a $(c, t)$-sparse graph with $n$ vertices, for some $c > 0$ and $t$, then there exists $C = C(c, H) > 0$ such that any subgraph $G \subseteq \Gamma$ with at least $Ct^{1/d}n^{2 - 1/d}$ edges contains a copy of $H$ which is induced in $\Gamma$.
\end{corollary}

A result of Alon, Krivelevich, and Sudakov \cite{AKSudakov} and Furedi \cite{Furedi} states that if $H$ satisfies the condition in Corollary \ref{cor:deg_d}, then $\mathrm{ex}(n, H) = O(n^{2 - 1/d})$. Therefore, Corollary \ref{cor:deg_d} can be seen as an induced analogue. Moreover, the counting part of Theorem \ref{thm:drc} provides an induced analogue of a result of Conlon, Fox, and Sudakov \cite{CFS10} on Sidorenko's conjecture. The proof of Theorem \ref{thm:drc} uses a variant of the dependent random choice motivated by the one from \cite{CFS10}, however we encounter a number of challenges that do not arise in the non-induced setup. As in \cite{CFS10}, we construct appropriate embeddings of the vertices in $B$ of $H$ by selecting vertices in the common neighborhood of an $\ell$-tuple of vertices. We then consider, for each vertex $a$ in $A$ of $H$ with degree at most $d$, a large set of potential options for embedding $a$. However, compared to the non-induced setting, in order to guarantee that we have the correct number of options viable for creating an induced copy of $H$, we need to make several crucial modifications of the dependent random choice argument. In particular, we cannot obtain the desired guarantee when selecting uniformly at random $\ell$ vertices and looking for the embedding of $B$ in their common neighborhood. Instead, an important idea in the proof is to choose each of the $\ell$ vertices one at a time, with each selection preceded by a degree regularization procedure.

Note that Corollary \ref{cor:deg_d} solves the problem of Alon, Krivelevich, and Samotij \cite{AKS23}. Indeed, if $\Gamma=G(n,p)$ with $0<p<1$ fixed, with high probability the number of edges of $G$ is  $(1+o(1))p{n \choose 2}$, and, as earlier remarked, $\Gamma$ is $(c,t)$-sparse for any fixed $c<1-p$ and any $t \geq C'\log n$ where $C'=C'(p,c)$ is sufficiently large. If $k(\mathcal{P})=2$, then there is a bipartite graph $H$ not in $\mathcal{P}$, and for $d$ being, say, the maximum degree of $H$, we can apply Corollary \ref{cor:deg_d} to deduce  
$\mathrm{ex}(\Gamma,\mathcal{P}) \leq n^{2-1/d}$. In particular, $\mathrm{ex}(\Gamma, \mathcal{P}) \le n^{2 - \eps}$ for some $\varepsilon=\varepsilon(\mathcal{P}) > 0$. We note that Clifton, Liu, Mattos and Zheng \cite{CLMZ} give an independent solution of the problem of Alon, Krivelevich and Samotij, building on the approach of \cite{BBCD}. One main difference is that in the present work, we are interested in deriving a sharp exponent $\eps(\cP)$ for various properties $\cP$. 

The rest of the introduction is organized as follows. In Subsection \ref{subsec:opt}, we discuss several cases in which the lower bound on the number of edges of $G$ in Theorem \ref{thm:drc} is best possible. In Subsection \ref{subsec:tree}, we discuss our result giving tight bounds on the size of a largest subgraph of a $(c,t)$-sparse graph forbidding an induced tree. In comparison to the ordinary Tur\'an problem for trees, the situation here becomes drastically more involved.  In Subsection \ref{subsec:stringsub}, we discuss an application of Theorem \ref{thm:drc} that obtains tight bounds on the size of a largest string subgraph of a $(c,t)$-sparse graph. We then discuss the relationship with the special case of forbidding a large complete bipartite subgraph in Subsection \ref{subsec:no-complete-bip}, which has been the subject of much recent interest. %

\subsection{Optimality of Theorem \ref{thm:drc}} \label{subsec:opt}

For a graph $H$, let $\mathcal{P}_H$ be the hereditary graph property of not containing $H$ as an induced subgraph. For $\Gamma$ a graph with $p{n \choose 2}$ edges, by averaging, we have $\mathrm{ex}(\Gamma,\mathcal{P}_H) \geq p \, \mathrm{ex}(n,H)$. Analogous to the case where $H$ is not bipartite, it is natural to suspect that this is a good bound, at least for $\Gamma = G(n,p)$. Surprisingly, we show that this is not the case. Our general lower bound on $\mathrm{ex}(\Gamma,\mathcal{P}_H)$ is related to the chromatic number of the complement of $\Gamma$. In what follows, we are mainly concerned with $H = K_{s,r}$ as $\mathrm{ex}(n, K_{s,r})$ is well understood for $r$ sufficiently large in $s$.

\begin{lemma} \label{lowercor}
 Let $2 \leq s \leq r$ be integers with $r > 2$. For every graph $\Gamma$, we have  
 $$
 \mathrm{ex}(\Gamma,\mathcal{P}_{K_{s,r}}) \geq \frac{1}{2}e(\Gamma)\mathrm{ex}(\chi(\overline{\Gamma}),K_{s,r})/{\chi(\overline{\Gamma}) \choose 2}.
 $$
\end{lemma}

We also provide a bound for a general graph $H$ (see Corollary \ref{lowercor-fam}). These are proved in Section \ref{sec:lower}.

Closely related to both $\chi(\overline{\Gamma})$ and $(c,t)$-sparseness is $\omega(\Gamma)$, the size of a largest clique in $\Gamma$. We have $\chi(\overline{\Gamma}) \ge n / \omega(\Gamma)$ and if $\Gamma$ is $(c, t)$-sparse for $c$ which is not too small, then $t \ge \omega(\Gamma)$. The latter follows from $e(A, A) = (1 - 1/|A|)|A|^2$ where $A$ is a largest clique in $\Gamma$. In the case of $\Gamma = G(n,p)$, for constant $0<p<1$, it is well-known that $\omega(\Gamma) = \Theta(\log n)$ and $\chi(\overline{\Gamma}) = (1 + o(1)) n / \omega(\Gamma)$ with high probability (see \cite{bollobas88}). Moreover, standard estimates on the edge distribution of $G(n,p)$ imply it is with high probability $(c, 2 \omega(\Gamma))$-sparse, for some $c=c(p)>0$. Combining Theorem \ref{thm:drc} and Lemma \ref{lowercor} with known estimates on $\mathrm{ex}(n, K_{s,r})$, we thus obtain the following.

\begin{theorem}\label{gnp}
Let $2 \leq s \leq r$ be fixed with $r$ sufficiently large in $s$.  
If $0<p<1$ is fixed and $\Gamma=G(n,p)$, then with high probability we have 
$\mathrm{ex}(\Gamma,\mathcal{P}_{K_{s,r}})=\Theta\left(n^{2-1/s}(\log n)^{1/s}\right)$. 
\end{theorem}

We also obtain tight bounds in the case $H = K_{s,r}$, again for $r$ sufficiently large in $s$, and $\Gamma$ is a Paley graph $P_q$ for $q$ a perfect square. For a prime power $q \equiv 1 \pmod 4$, recall that the Paley graph $P_q$ is a graph whose vertex set is the elements of a finite field with $q$ elements, and two distinct vertices are adjacent if their difference is a perfect square. The clique number of $P_q$ is known to be at most $\sqrt{q}$, with equality if $q$ is a perfect square. Unlike in the case of random graphs, we do not know that $\chi(\overline{\Gamma}) = \Theta(n / \omega(G))$, but we can find a large subgraph $\Gamma' \subseteq \Gamma$ for which this holds and then apply Lemma \ref{lowercor} with $\Gamma'$. We defer the details to Section \ref{sec:lower}. Standard estimates on character sums (see~Theorem 3.4.2 of \cite{Yip}) implies that for each $c<1/2$ there is $C=C(c)$ such that every Paley graph with $n$ vertices is $(c,C\sqrt{n})$-sparse. Combined, this gives the following.

\begin{theorem}\label{paley}
Let $2 \leq s \leq r$ be fixed with $r$ sufficiently large in $s$.
If $\Gamma$ is a Paley graph with a square number $n$ of vertices, then $\mathrm{ex}(\Gamma,\mathcal{P}_{K_{s,r}})=\Theta(n^{2-1/(2s)})$.
\end{theorem}

\subsection{Induced trees in $(c,t)$-sparse graphs}
\label{subsec:tree}

For a fixed tree $T$, we consider the property $\mathcal{P}_T$ of not containing $T$ as an induced subgraph. The following is our second main result.

\begin{theorem} \label{thm:tree}
    For any tree $T$ and $c > 0$, there exists $C > 1$ such that the following holds. Given a $(c, t)$-sparse graph $\Gamma$ on $n$ vertices, for some $t$, every subgraph $G \subseteq \Gamma$ with at least $Ctn$ edges contains a copy of $T$ which is induced in $\Gamma$. 
\end{theorem}

Assuming that $T$ contains at least two edges, the disjoint union of cliques is an induced-$T$-free graph. A graph $\Gamma$ that is a disjoint union of cliques of size $t/4$ is $(1/2, t)$-sparse, thus the bound on the number of edges required in Theorem \ref{thm:tree} is best possible. For $\Gamma=G(n,p)$, by taking a subgraph which is a disjoint union of cliques, it follows that $\mathrm{ex}(\Gamma,\mathcal{P}_{T})=\Omega(n\log_{1/p}n)$ with high probability. Indeed, we have $\chi(\overline{\Gamma}) = O(n/\log_{1/p} n)$ with high probability (see \cite{bollobas88}), which gives a spanning subgraph of $\Gamma$ which is a disjoint union of $ O(n/\log_{1/p} n)$ cliques. A standard application of convexity shows that such a subgraph has at least  $\Omega(n\log_{1/p}n)$ many edges. In particular, this shows $\ex(\Gamma,\mathcal{P}_T) = \Theta(n\log n)$ when $0<p<1$ is a constant. Similarly, when $\Gamma$ is a Paley graph $P_q$ for $q$ a perfect square, the same argument together with Lemma \ref{vertextransitive} shows that we have $\ex(\Gamma,\mathcal{P}_T) = \Theta(n^{3/2})$. 

A standard way of proving that the Tur\'an number of a fixed tree $T$ is $\Theta(n)$ proceeds by passing on to a subgraph with sufficiently large (constant) minimum degree which depends on $T$, and then constructing a copy of $T$ greedily. However, we face significant challenges in guaranteeing that the tree is induced under the $(c,t)$-sparse condition particularly when the graph $G$ has irregular degree distribution. In order to address this, our proof of Theorem \ref{thm:tree} follows from a fairly involved procedure to glue together \emph{distributions} $\lambda_{T'}$ supported on induced copies of subtrees $T'$ of $T$. This is our main technical contribution, thus we give a brief overview of the strategy.

The crucial property satisfied by the distributions $\lambda$ is that, for a tree $T'$ obtained by removing a leaf from $T$, $\lambda_{T}$ and $\lambda_{T'}$ satisfy appropriate consistency relations. The main step is then to inductively construct, for subtrees $T_0$ of $T$, the measure $\lambda_{T_0}$ from $\lambda_{T'}$ for smaller subtrees $T'$ of $T_0$. We identify two subtrees $T_1,T_2$ of $T_0$ such that each $T_1,T_2$ is obtained from $T_0$ by removing a (different) leaf. Let $T'$ be the intersection of $T_1$ and $T_2$. We first sample an embedding $\rho_{T'}$ of $T'$ from $\lambda_{T'}$. Granted that $\rho_{T'}$ lies in a good event, we consider two independent extensions of $\rho_{T'}$ to embeddings of $T_1$ and $T_2$, sampled according to $\lambda_{T_1}$ and $\lambda_{T_2}$ conditional on the image of $T'$ being $\rho_{T'}$. If combining both extensions gives us a good embedding of $T_0$, then we successfully obtain a sample defining $\lambda_{T_0}$. Using consistency relations between $\rho_{T'}$ and $\rho_{T_1}$, and between $\rho_{T'}$ and $\rho_{T_2}$, and the $(c,t)$-sparse assumption, we can show that we successfully obtain a sample with a good probability. We then close the induction by verifying that the consistency relations hold between the defined distribution $\lambda_{T_0}$, and the distributions $\lambda_{T_1}$ and $\lambda_{T_2}$.  

\subsection{Largest induced string subgraphs}\label{subsec:stringsub}

Returning to the original set-up of Alon, Krivelevich, and Samotij, beyond studying the case of a single forbidden induced subgraph, estimating $\textrm{ex}(\Gamma,\mathcal{P})$ remains an interesting problem when not all bipartite graphs belong to $\mathcal{P}$.

A string graph is an intersection graph of arcwise connected sets in the plane. Let $\mathcal{S}$ denote the family of string graphs. Note that not every graph is a string graph. Indeed, subdividing each edge of a nonplanar graph at least once produces a graph which is not a string graph. In particular, subdividing each edge of the complete graph on five vertices yields a bipartite graph on fifteen vertices which is not a string graph. A nontrivial upper bound on $\textrm{ex}(\Gamma,\mathcal{S})$ for $\Gamma$ a $(c,t)$-sparse graph follows from Theorem \ref{thm:drc}. However, we prove the following much better upper bound using an additional property of string graph, which roughly says that any string graph has a dense induced subgraph with about the same average degree.  

\begin{theorem}\label{string}
For every $c > 0$ there exists $C$, such that if $\Gamma$ is a $(c,t)$-sparse graph on $n$ vertices then $\mathrm{ex}(\Gamma, \mathcal{S}) \le Ctn$.
\end{theorem}

Similarly as in the case of trees, this bound is tight when $\Gamma$ is a random graph or a Paley graph with a square number of vertices. As an application of Theorem \ref{string}, we discuss in Section \ref{sec:string} the determination of $\ex(\Gamma,\calP)$ for $\calP$ the property of being an incomparability graph, and $\Gamma$ a random graph or a Paley graph with a square number of vertices. 

\subsection{Forbidding large complete bipartite subgraphs}\label{subsec:no-complete-bip}

A related induced Tur\'an number was introduced and studied by Loh et al.~\cite{LTTZ}. For a positive integer $n$ and graphs $F$ and $H$, let $\textrm{ex}(n,F,H\textrm{-ind})$ be the maximum number of edges in a graph on $n$ vertices which does not contain $F$ as a subgraph and $H$ as an induced subgraph. If $F$ is bipartite with parts of size $r \le s$ and has $m$ edges, then any graph not containing $F$ as a subgraph is $(c, t)$-sparse with $c = 1/(2m)$ and $t = r + s$. Indeed, suppose that $\Gamma$ is not $(c,t)$-sparse and let $(A, B)$ be a witness for this. Then by first randomly embedding a smaller part of $F$ into $A$, and then randomly embedding a larger part into the remaining free vertices in $B$ (of which there are at least $|B| - r \ge |B|/2$), we have that a particular edge of $F$ is not present in such an embedding with probability less than $2c = 1/m$, thus $\Gamma$ contains $F$. 

Furthermore, if $F=K_{t,t}$ is a balanced complete bipartite graph, then it is not difficult to show that not containing $K_{t,t}$ as a subgraph is equivalent to being $(c,t)$-sparse with $c=1/t^2$. Several very recent papers \cite{AZ, BBCD,GM,HMST, Zim} prove upper bounds on $\textrm{ex}(n,K_{t,t},H\textrm{-ind})$ for various graphs $H$. The bounds from \cite{HMST} establish 
that $\textrm{ex}(n,K_{t,t},H\textrm{-ind})=t^{O(|V(H)|)}n^{2-1/d}$ if $H$ is bipartite and every vertex in one part of $H$ has maximum degree at most $d$. The proof of Theorem \ref{thm:drc} shows that we may take $C=c^{-O(|V(H)|^2)}$, and hence implies that $\textrm{ex}(n,K_{t,t},H\textrm{-ind})=t^{O(|V(H)|^2)}n^{2-1/d}$ if $H$ is $d$-bounded. While it is plausible that the methods from \cite{HMST} could also be extended to the case of $d$-bounded graphs, the main message here is that our bounds obtained for the more general setup of $(c,t)$-sparse graphs also give reasonably good bounds in the special case of $K_{t,t}$-free graphs.

Note that, for $0<p<1$ fixed, the Erd\H{o}s-Renyi random graph $G(n,p)$ almost surely does not contain $K_{t,t}$ as a subgraph for a $t$ which is logarithmic in $n$. Thus the main result of Hunter, Milojevi\'c, Sudakov, and Tomon \cite{HMST} gives a weaker upper bound (off by a polylogarithmic in $n$ factor) on $\mathrm{ex}(\Gamma, \mathcal{P}_{K_{s,r}})$ with $\Gamma=G(n,p)$ for $0<p<1$ fixed, compared to Theorem \ref{gnp}. For Paley graphs with a square number $n$ of vertices, which contains a clique of order $\sqrt{n}$, these results do not give a nontrivial upper bound on induced Tur\'an numbers. In general, our setup is a natural strengthening where instead of forbidding large complete bipartite subgraphs we forbid large dense subgraphs. One attractive feature of this strengthening is that, in some cases, it allows us to prove bounds which are optimal in the size of such subgraphs. Consequently, this gives optimal bounds in the case the host graph is, say, a dense random graph. With this in mind, our focus was on optimizing the dependence on $t$. We did not try to optimize, or even calculate, the dependence on $c$.

\textbf{Organization.} In Section \ref{sec:lower}, we describe several constructions giving lower bounds on $\ex(\Gamma, \calP_H)$. In Section \ref{sec:main}, we prove Theorem \ref{thm:drc} based on a novel variation of dependent random choice. In Section \ref{sec:tree}, we give the proof Theorem \ref{thm:tree} on embedding induced trees in subgraphs of $(c,t)$-sparse graphs. In Section \ref{sec:string}, we discuss applications of Theorem \ref{thm:drc} to the extremal numbers of the properties of being string graphs, incomparability graphs and comparability graphs. In Section \ref{sec:concluding}, we finish with some concluding remarks, highlighting open problems and discussing consequences of our results for induced Ramsey numbers. 

\textbf{Notation.} We follow standard graph theoretic notation. Given a graph $G$ and a subset $S \subseteq V(G)$, we denote with $N_G^+(S)$ the union of neighborhoods of vertices in $S$, that is
$$
    N_G^+(S) = \bigcup_{v \in S} N_G(v).
$$
Similarly, we let $N_G(S)$ denote the \emph{common neighborhood} of vertices in $S$, that is
$$
    N_G(S) = \bigcap_{v \in S} N_G(v).
$$

\textbf{Acknowledgments.} Part of this work was completed while the second author was visiting Stanford University. We would also like to thank Matija Buci\'c for helpful discussions on connections to the recent works on induced Tur\'an numbers. We thank Jakob Zimmermann for carefully reading the paper and suggesting improvements. We also thank Maria Axenovich for helpful comments. 

\section{Lower bound on a largest subgraph in $\mathcal{P}_H$} \label{sec:lower}

Recall that any graph contains a bipartite subgraph with at least half of the edges. Therefore, Lemma \ref{lowercor} follows immediately from the following by taking the maximum bipartite subgraph $F$ of a $K_{s,r}$-free graph on $\chi(\overline{\Gamma})$ vertices and $\mathrm{ex}(\chi(\overline{\Gamma}),K_{s,r})$ edges.

\begin{lemma}\label{lowerlemma}
Let $\Gamma$ be a graph such that $\chi(\overline{\Gamma}) = k$. Suppose $F$ is a bipartite graph on $k$ vertices which is $K_{s,r}$-free (as a subgraph), for some $2 \leq s \leq r$ such that $r \not = 2$.  Then $\Gamma$ has a subgraph $G$ with at least $e(\Gamma)e(F)/{k \choose 2}$ edges that is induced $K_{s,r}$-free. 
\end{lemma}
\begin{proof}
Since $\chi(\overline{\Gamma})= k$, we can vertex partition $\Gamma$ into $k$ cliques $A_1,\ldots,A_k$. Consider a uniform random bijection $f$ from $\{1,\ldots,k\}$ to $V(F)$. Let $G$ be the subgraph of $G$ consisting of all edges that are in the cliques $A_i$ and any edge of $\Gamma$ between $A_i$ and $A_j$ with $(f(i),f(j)) \in E(F)$.

We first show that $G$ is necessarily induced $K_{s,r}$-free. Indeed, suppose there was an induced copy of $K_{s,r}$ in $G$. If each of the $s+r$ vertices are in distinct parts $A_i$, then the corresponding vertices of $F$ would form a copy of $K_{s,r}$, thus a contradiction. No two vertices of an induced copy of $K_{s,r}$ from the same part in the bipartition can lie in the same $A_i$ as $A_i$ is a clique but they must be nonadjacent. So there must be a vertex from each part that maps to the same $A_i$, say $v_1$ and $w_1$ map to $A_1$. As $2 \leq s \leq r$ and $r > 2$, we can choose $v_2, v_3$ (distinct from $v_1$) from one part and $w_2$ (distinct from $w_1$) from the other part in the bipartition of $H$. As all three $v_1, v_2$, and $v_3$ have to lie in different $A_j$'s, and $w_2$ cannot lie in $A_1$, we obtain three different $A_j$'s which are pairwise connected by an edge. This implies $F$ contains a triangle, contradicting that $F$ is bipartite. 

Each edge of $\Gamma$ has probability at least $e(F)/{k \choose 2}$ of surviving in $G$, and so the expected number of edges of $G$ is at least $e(\Gamma)e(H)/{k \choose 2}$. Hence there is a choice of such a $G$ with at least this many edges. 
\end{proof}

We remark that the bipartite assumption on $H$ in Lemma \ref{lowerlemma} is stronger than actually needed, but for an extremal problem, this assumption does not make a big difference, only changing the answer by a factor at most $1/2$. 

As remarked earlier, for a graph $\Gamma$ with $n$ vertices we have $\chi(\overline{\Gamma}) \geq n/\omega(\Gamma)$. This bound is asymptotically sharp with high probability for $\Gamma=G(n,p)$ with $0<p<1$ fixed. In the case $\Gamma$ is a vertex-transitive graph, by randomly picking out largest cliques we get that this lower bound on $\chi(\overline{\Gamma})$ is sharp up to a logarithmic factor. In order to get an optimal bound in Theorem \ref{paley}, it suffices to cover most of the vertices instead of all (or, more importantly, covered vertices need to span many edges).

\begin{lemma}\label{vertextransitive}
Let $\Gamma$ be a vertex transitive graph on $n$ vertices. Then $\Gamma$ has an induced subgraph $\Gamma'$ with at least $e(\Gamma)/2$ edges satisfying $\chi(\overline{\Gamma'}) \leq (\ln 4)n/\omega(\Gamma)$. 
\end{lemma}
\begin{proof}
Since $\Gamma$ is vertex transitive, it is $d$-regular for some $d$. Let $A_1, \ldots, A_k$ be cliques of size $\omega = \omega(\Gamma)$ in $\Gamma$ sampled uniformly and independently. Each vertex belongs to the same number of cliques of size $\omega$, thus the probability that a particular vertex does not belong to one $A_i$ is exactly $1 - \omega / n$. By the independence of $A_i$'s we have that it does not belong to any $A_i$ with probability $(1 - \omega/n)^k$. Using the linearity of expectation, we conclude
$$
    \mathbb{E}\left[ |V(\Gamma) \setminus (A_1 \cup \ldots \cup A_k)| \right] = n (1  - \omega / n)^k.
$$
Therefore, for $k = (\ln 4) n / \omega$ there exist cliques $A_1, \ldots, A_k$ which cover all but at most $n/4$ vertices. 
Moreover, the subgraph $\Gamma'$ induced by $A_1 \cup \ldots A_k$ can be partitioned into $k$ cliques, thus $\chi(\overline{\Gamma'}) \le k$. The vertices not in $\Gamma'$ are incident to at most $dn/4$ edges, so $e(\Gamma') \geq e(\Gamma)-dn/4 =e(\Gamma)/2$.
\end{proof}

For vertex transitive graphs $\Gamma$, we obtain a lower bound on $\mathrm{ex}(\Gamma, \mathcal{P}_{K_{s,r}})$ by applying Lemma \ref{lowercor} to the induced subgraph $\Gamma'$ of $\Gamma$. Crucially, by passing on to $\Gamma'$ we lose only a constant fraction of all the edges.

\begin{proof}[Proof of Theorem \ref{paley}]
    Let $\Gamma = P_q$ for a square $q$. In this case it is known that $\omega(\Gamma) = \sqrt{q} = \sqrt{n}$. As a Paley graph is a Cayley graph, and these are known to be vertex transitive, by applying Lemma \ref{vertextransitive} we obtain a subgraph $\Gamma' \subseteq \Gamma$ with $e(\Gamma') \ge n(n-1)/8$ edges such that $\chi(\overline{\Gamma'}) \le 2 \sqrt{n}$. By applying Lemma \ref{lowercor} on $\Gamma'$ instead of $\Gamma$, we obtain
    $$
        \mathrm{ex}(\Gamma, \mathcal{P}_{K_{s,r}}) \ge \mathrm{ex}(\Gamma', \mathcal{P}_{K_{s,r}}) = \Omega(n^{2 - 1/(2s)}). \qedhere
    $$
\end{proof}

\subsection{General graphs $H$}

For a graph $H$ and a vertex partition $\mathcal{P}:V(H)=V_1 \sqcup \ldots \sqcup V_k$, the quotient graph of $H$ obtained from the partition $\mathcal{P}$ is the graph with vertex set $\{1,\ldots,k\}$ and where $(i,j)$ with $i \not = j$ is an edge if some vertex in $V_i$ is adjacent to a vertex in $V_j$. That is, the quotient graph is the graph obtained from $H$ by contracting each of the sets $V_i$. If each $V_i$ is a clique in $H$, then this is a {\it clique quotient of $H$}. For each graph $H$, let $\mathcal{F}_H$ denote the family of bipartite graphs that are clique quotients of $H$. Note that if $H$ is bipartite, then $H \in \mathcal{F}_H$ as $H$ is a clique quotient of itself by taking the vertex partition into singletons. %
Furthermore, for bipartite $H$, the family $\mathcal{F}_H$ consists of exactly the bipartite graphs that can be obtained by contracting disjoint edges of $H$. The proof of Lemma \ref{lowerlemma} extends directly to show the following generalization. 

\begin{lemma}\label{lem:lowerlem-gen}
Let $\Gamma$ be a graph such that $\chi(\overline{\Gamma}) = k$. Let $H$ be a graph and suppose $F$ is a bipartite graph on $k$ vertices which does not contain any graphs in $\mathcal{F}_H$ as a subgraph.  Then $\Gamma$ has a subgraph $G$ with at least $e(\Gamma)e(F)/{k \choose 2}$ edges that is induced $H$-free. 
\end{lemma}

Lemma \ref{lem:lowerlem-gen} immediately implies the following generalization of Lemma \ref{lowercor}.
\begin{corollary}\label{lowercor-fam}
    For all graphs $H$ and $\Gamma$, we have  $$\mathrm{ex}(\Gamma,\mathcal{P}_{H}) \geq \frac{1}{2}e(\Gamma)\mathrm{ex}(\chi(\overline{\Gamma}),\mathcal{F}_H)/{\chi(\overline{\Gamma}) \choose 2}.$$ 
\end{corollary}

We refer the reader to a discussion in Section \ref{sec:concluding} where we make use of this corollary to give a lower bound on the number of edges when the forbidden subgraph is an even cycle $C_{2k}$.

\section{Counting $d$-bounded graphs} \label{sec:main}

In this section we prove Theorem \ref{thm:drc}. The following two lemmas will be useful to guarantee that the nonedges in the copy of $H$ we find in a subgraph of the $(c,t)$-sparse graph $\Gamma$ are nonedges of $\Gamma$.

\begin{lemma}\label{lemma:large_nonnbr}
Let $\Gamma$ be a $(c,t)$-sparse graph with $c \leq 1/2$. Let $V, W \subseteq V(\Gamma)$ with $|W| \geq c^{-r}t$, for some positive integer $r$. If we pick a uniformly random $r$-tuple $R$ of vertices in $V$, then
$$
 \Pr_R\left[ |W \setminus N_\Gamma^+(R)| < c^r |W| \right] \le rt / |V|.
$$
\end{lemma}
\begin{proof}
We pick vertices in a tuple $R$ one at a time, picking $v_i$ in step $i$. Let $W_0=W$ and after picking $v_i$, let $W_i = W \setminus N_\Gamma^+(\{v_1, \ldots, v_i\})$. Assuming $|W_i| \geq c^{i}|W| \geq t$, the number of vertices in $V$ that do not have at least $c|W_i|$ nonneighbors in $W_i$ is at most $t$. Hence, the probability that $|W_{i+1}| < c|W_i|$ conditioned on $|W_i| \geq c^{i}|W|$ is at most $t/|V|$. Therefore, the probability that the number of common nonneighbors of $R$ is less than $c^{-r}|W|$ is at most $rt/|V|$.
\end{proof}

\begin{lemma}\label{lemma:independent_set}
Let $\Gamma$ be a $(c,t)$-sparse graph with $0<c \leq 1/2$. Let $k \geq 2$ and $V_1,\ldots,V_k$ be vertex subsets of $\Gamma$, not necessarily disjoint, each of size at least $m \geq k^2 c^{-k^2} t$. A uniformly random $k$-tuple in $V_1 \times \cdots \times V_k$ forms an independent set of $k$ distinct elements with probability at least $c^{k^2}$. 
\end{lemma}
\begin{proof}

Consider picking the vertices one at a time, picking $v_i \in V_i$ in step $i$. Having picked $v_1,\ldots,v_i$ for some $0 \le i < k$, let $V_{j,i}$ for $i < j \leq k$ be the set of vertices in $V_j$ not adjacent to $v_1,\ldots,v_i$. Note that $V_{j,0} = V_j$ for every $j$. For $0 \le i < k$, call $v_{i+1} \in V_{i+1}$ \emph{good} if the number of nonneighbors of $v_{i+1}$ in $V_{j,i}$ is at least $c|V_{j,i}|$ for all $i + 1 < j \le k$. Note that if $|V_{j,i}| \geq t$, then all but at most $t-1$ vertices $v_{i+1} \in V_{i+1}$ have at least $c|V_{j,i}|$ nonneighbors in $V_{j,i}$. Hence, assuming $|V_{j,i}| \geq t$ for $i+1 <j \leq k$, there are at most $(k-i)(t-1)$ vertices in $V_{i+1}$ which are not good. Note that if each $v_1, \ldots, v_i$ is good, then $|V_{j,i}| \geq c^i |V_j| \geq t$ for $i+1 \leq j \leq k$. Hence the probability that not all $v_1,\ldots,v_{k-1}$ are good is at most 
$$
\sum_{i=1}^{k-1} \frac{(k-i)(t-1)}{m} = \binom{k}{2} (t-1)/m.
$$
For $1 \leq i \leq k-1$, if $v_1,\ldots,v_{i}$ are good, then $|V_{i+1,i}| \geq c^i |V_{i+1}|$. So conditioning on $v_1,\ldots,v_{k-1}$ being good, the probability that there are no edges between $v_1,\ldots,v_k$ is at least $\prod_{i=1}^k c^{i-1} = c^{{k \choose 2}}$. 

For $1 \leq i < j \leq k$, the probability that the same vertex is picked in $V_i$ and $V_j$ is at most $1/m$. Hence by the union bound, the probability that uniform random $v_1,\ldots,v_k$ are all distinct is at least $1 - {k \choose 2}/m$. All together, the probability that the chosen vertices form an independent set of distinct vertices is at least 
$$
    c^{\binom{k}{2}} - \binom{k}{2}t/ m \geq c^{k^2}.
$$
\end{proof}

Given a graph $G$ and a subset $S \subseteq V(G)$, we denote with $d_\mathrm{avg}(S)$ the average degree of vertices in $S$.  We say that a bipartite graph $G$ on vertex sets $U$ and $R$ is $\delta$-reduced if $d(u) \in [\delta d_\mathrm{avg}(U), \delta^{-1}d_\mathrm{avg}(U)]$ for all $u \in U$. The following lemma provides a strong form of one-sided regularization.

\begin{lemma}\label{lem:reduce}
    Given $\kappa > 1$, let $0 < \delta < 1/4$ be so that $1 - 2\delta > \delta^{1 - 1/\kappa}$. Let $G$ be a bipartite graph with vertex parts $U$ and $R$. Then there exists $U'\subseteq U$ such that $G' = [U', R]$ is $\delta^2$-reduced, 
    $$
        d_\mathrm{avg}(U') \ge \delta d_\mathrm{avg}(U),
    $$
    and for any real $x \ge \kappa$ we have 
    $$
        |U'| d_\mathrm{avg}(U')^x \ge \delta^x |U| d_\mathrm{avg}(U)^x.
    $$
\end{lemma}
\begin{proof}
    Set $U_0=U$, and consider the following process for $i\ge 0$. As long as 
    \begin{equation} \label{eq:continue_condition}
        \sum_{\substack{u\in U_i \\ d(u) \le \delta^{-1}D_i}} d(u) < 2\delta e(U_i, R),
    \end{equation}
    where $D_i = d_{\mathrm{avg}}(U_i)$ is the average degree of vertices in $U_i$, set 
    $$
        U_{i+1} := \{u\in U_i: d(u) \ge \delta^{-1}D_i\}
    $$
    and proceed to the next iteration. By Markov's inequality, we have $|U_{i+1}| \le \delta |U_i|$. By \eqref{eq:continue_condition}, we have $|e(U_{i+1}, R)| \ge (1-2\delta)|e(U_i, R)|$. 

    Let $i_0$ denote the first index $i$ for which \eqref{eq:continue_condition} does not hold. Set 
    $$
        U'=\{u\in U_{i_0}:\delta D_{i_0} \le d(u) \le \delta^{-1}D_{i_0}\}.
    $$
    Note that $G' = G[U', R]$ is $\delta^2$-reduced. Furthermore, from the assumption that \eqref{eq:continue_condition} fails for $i_0$ we have $e(U', R) \ge \delta e(U_{i_0}, R)$. We also have that $e(U_{i_0}, R) \ge (1-2\delta)^{i_0} e(G)$ and $|U'| \le |U_{i_0}| \le \delta^{i_0}|U|$. Then
    $$
        d_\mathrm{avg}(U') = \frac{e(U', R)}{|U'|} \ge \frac{\delta e(U_{i_0}, R)}{|U_{i_0}|} \ge \delta \left( \frac{1 - 2\delta}{\delta} \right)^{i_0} d_\mathrm{avg}(U) \ge \delta d_\mathrm{avg}(U).
    $$    
    Similarly, for $x \ge \kappa$ we have
    $$
        |U'| d_{\mathrm{avg}}(U')^x = \frac{e(U', R)^x}{|U'|^{x - 1}} \ge \frac{\delta^x e(U_{i_0}, R)^x}{|U_{i_0}|^{x-1}} \ge \delta^x \left( \frac{1 - 2\delta}{\delta^{1-1/x}} \right)^{x i_0} |U| d_\mathrm{avg}(U)^x.
    $$
    As $1 - 2\delta > \delta^{1 - 1/\kappa} \ge \delta^{1 - 1/x}$, we obtain the desired inequality.
\end{proof}

Given a graph $H$, for the rest of the section we fix $\kappa = 1 + 1/v(H)$ and let $\delta$ be as in \Cref{lem:reduce}.

The following lemma is the heart of the proof of Theorem \ref{thm:drc}. Given a graph $\Gamma$ with $n$ vertices and a bipartite subgraph $G \subseteq \Gamma$ on vertex sets $U$ and $R$ with $|U|=m,|R|=r$, let $q = e(G)/(mr)$ be the edge density of $G$. We say that $S \subseteq R$ is \emph{$(\varepsilon, T)$-rich}, for some $\varepsilon > 0$ and $T \subseteq U$, if
$$
    |N_G(S) \setminus N_\Gamma^+(T)| \ge \varepsilon m q^{|S|}.
$$
For $1\le k\le d$, we say that $T$ is \emph{$(\beta, \varepsilon, k)$-good} if the number of $k$-tuples of vertices in $N_G(T)$ that are not $(\varepsilon, T)$-rich is at most $\beta |N_G(T)|^k$. Note that we do not require elements in these tuples to be distinct.

\begin{lemma} \label{lemma:drc_count_1}
For every $0 < c < 1/2$, $\beta > 0$, and positive integer $d$ there are $C, \varepsilon > 0$ such that the following holds. Let $\Gamma$ be a $(c,t)$-sparse graph with $n$ vertices and $G \subseteq \Gamma$ be a $\delta$-reduced bipartite subgraph of $\Gamma$ with parts $U,R$ with $|U|=m,|R|=r$, and at least $C rm^{1 - 1/d} t^{1/d}$ edges. Then, with probability at least $1/2$ over uniformly chosen $u\in U$, we have that $u$ is $(\beta,\eps,k)$-good for all $k\le d$. 
\end{lemma}
\begin{proof}
    Consider $k \in [d]$. Let $\mathcal{B}_k$ denote the set of all vertices $u$ which are not $(\beta, \varepsilon, k)$-good, where $\varepsilon > 0$ will be chosen as a sufficiently small constant in terms of $\beta$, $c$ and $d$.

    We distinguish two types of $k$-tuples that are not $(\eps,u)$-rich. Let $\mathcal{S}$ denote the family of $k$-tuples $S$ in $R$ such that $|N_G(S)| \ge c^{-1} \eps mq^k$. Let $s_1(u)$ denote the number of $k$-tuples $S\subset N_G(u)$ for which $S$ is not $(\eps,u)$-rich and $S\in \mathcal{S}$. Let $s_2(u)$ denote the number of $k$-tuples $S\subset N_G(u)$ for which $S$ is not $(\eps,u)$-rich and $S\notin \mathcal{S}$. 

    We first bound $\mathbb{E}[s_1(u)]$. For each $S\in \mathcal{S}$, since $\Gamma$ is $(c,t)$-sparse there are at most $t$ vertices $u$ for which $S$ is not $(\eps,u)$-rich. Hence, 
    \[
    \mathbb{P}[\textrm{$S$ is not $(\eps,u)$-rich} \mid S\subset N_G(u)] \le \frac{t}{|N_G(S)|} \le \frac{t}{c^{-1} \eps mq^k}. 
    \]
    In particular, since $G$ is $\delta$-reduced,
    \[
    \frac{c^{-1} \eps mq^k}{t}\mathbb{E}[s_1(u)] \le \mathbb{E}[|N_G(u)|^k] \in [(\delta D)^k, (\delta^{-1}D)^k],
    \]
    where $D$ is the average degree of vertices in $U$. Next, we bound $\mathbb{E}[s_2(u)]$. For $S\notin \mathcal{S}$ we have 
    \[
    \mathbb{P}[S\subset N_G(u)] \le c^{-1}\eps q^k.
    \]
    Hence, 
    \[
    \mathbb{E}[s_2(u)] \le r^kc^{-1}\eps q^k \le c^{-1} \eps (\delta^{-1} D)^k.
    \]

    By Markov's inequality, with probability at least $1-1/(2d)$ over $u\in U$, 
    \begin{align*}
        s_1(u)+s_2(u) &\le \frac{2dt}{c^{-1}\eps mq^k} (\delta^{-1}D)^k + 2d c^{-1} \eps (\delta^{-1} D)^k \le 2d (\delta^{-1}D)^k \left(\frac{ct}{\eps mq^k} + c^{-1}\eps \right). 
    \end{align*}
    For $\eps$ sufficiently small and $C$ sufficiently large, we then have $s_1(u)+s_2(u) \le \beta |N_G(u)|^k$ with probability at least $1-1/(2d)$. 

    By the union bound, with probability at least $1/2$ we have that $u$ is $(\beta,\eps,k)$-good for all $k\in [1,d]$. 
\end{proof}

Let $H = (A \cup B, E)$ be a $d$-bounded bipartite graph with $A_0 \subseteq A$ denoting the set of vertices complete to $B$ and $A_1 = A \setminus A_0$. Recall that, by the definition, every vertex in $A_1$ has degree at most $d$. Let $\ell = |A_0|$, $a = |A_1|$, $b = |B|$, and let $e$ be the number of edges of $H$. %

\begin{claim}  \label{claim:copies_T}
Assume $\ell=1$. Let $G$ be a bipartite subgraph of $\Gamma$ with vertex sets $|U|=m,|R|=r$. Assume that $G$ is $\delta$-reduced, and $r\ge |V(H)|^2c^{-|V(H)|^2}\delta^{-1} t(m/t)^{1/d}$. Let $\beta>0$ be sufficiently small in $H$, $c$ and $d$. Assume that $G$ contains at least $Crm^{1-1/d}t^{1/d}$ edges for $C$ sufficiently large in $\beta$, $H$, $c$ and $d$. Suppose $u \in U$ is $(\beta, \varepsilon, k)$-good. Then, for some constant $\gamma = \gamma(\delta, c, \beta, \eps, H) > 0$, there are at least 
$$
    \gamma m^{a} r^{b} q^{e(H)}
$$
copies of $H$ in $G$ which are induced in $\Gamma$ and contain $u$, where $q = e(G) / (mr)$ is the edge density of $G$.
\end{claim}
\begin{proof}
Consider a vertex $u\in U$ which is $(\beta,\eps,k)$-good for all $k\le d$. We then consider embeddings of $H$ where $A_0$ (of size $1$) maps to $u$. We then embed $B$ according to a uniformly random $b$-tuple $X\subseteq N_G(u)$. We first show that with decent probability, $X$ satisfies the following properties: 
\begin{enumerate}
    \item all elements in $X$ are distinct and independent in $\Gamma$,
    \item every $d$-tuple $Y$ in $X$ is $(c^{b-d} \eps, \{u\} \cup (X \setminus Y))$-rich. 
\end{enumerate}
The first property holds with probability at least $c^{b^2}$ by Lemma \ref{lemma:independent_set}. We show the second property holds with probability close to one by a union bound over $Y$. By the assumption that $u$ is $(\beta, \eps, k)$-good for $k\le d$, one particular $k$-tuple $Y$ in $N_G(u)$ fails to be $(\varepsilon, u)$-rich with probability at most $\beta$. Conditioning on $Y$ being $(\varepsilon, u)$-rich, by Lemma \ref{lemma:large_nonnbr} the probability of choosing the remaining vertices in $X$ from $N_G(u)\setminus Y$ such that 
$$
    |N_G(Y) \setminus N^+_\Gamma(\{u\} \cup (X \setminus Y))| < c^{b - k} \varepsilon mq^k
$$
is at most $(b-k)t / |N_G(u)\setminus Y|$. Therefore, one particular $Y$ fails to be $(\eps, \{u\} \cup (X \setminus Y))$-rich with probability at most $\beta + 2 b t / (\delta Cr(t/m)^{1/d})$. By the assumed lower bound on $r$ and by taking $C$ to be sufficiently large, we can upper bound this probability by $2\beta$. %
Finally, for $\beta > 0$ sufficiently small, by taking union bound over all $d$-tuples of $X$ we have that $X$ fails either property with probability at most
$$
    1 - c^{b^2} + \binom{b}{d} 2 \beta < 1 - c^{b^2} / 2.
$$

Next, we show that each $X$ which satisfies the two properties gives rise to $c^{a^2+ab} \eps^a m^a q^{e(H)-b}$ copies of $H$ in $G$ which are induced in $\Gamma$. As there are $c^{b^2} |N_G(u)|^b / 2 \ge (c^{b^2}/2) (\delta q r)^b$ such $b$-tuples $X$, this implies the claim.

Fix an ordering $v_1, \ldots, v_b$ of the vertices in $B$, and consider one such $X$. For every tuple of vertices $Y$ in $X$, let $U(Y) = N_G(Y) \setminus N^+_\Gamma(\{u\} \cup (X \setminus Y))$ denote the `unique' neighborhood of $Y$. Let $f$ be an arbitrary mapping of $v_1, \ldots, v_b$ to $X$ and, for each $w \in A_1$, let $B_w = f(N_H(w))$ and $V_w = U(B_w)$. Now pick one vertex from each $V_w$ for $w \in A_1$, uniformly at random. As $|V_w| \ge c^b \eps m q^{|N_H(w)|}$, by Lemma \ref{lemma:independent_set} we have that all chosen vertices are distinct and independent in $\Gamma$ with probability $c^{a^2}$. Assuming this is the case, we obtain a copy of $H$ in $G$ which is induced in $\Gamma$. Therefore, there are at least $c^{a^2}(c^b \eps m)^a q^{e(H)-b}$ copies which contain $u$.
\end{proof}

Theorem \ref{thm:drc} follows immediately from Lemma \ref{lem:count} below, upon replacing $G$ with a bipartite subgraph with at least $(1-o(1))e(G)/2$ edges on parts $U$ and $R$ with $|U|= \lfloor n/2 \rfloor$ and $|R| = \lceil n/2 \rceil$. 

\begin{lemma}\label{lem:count}
    For every $0 < c < 1/2$, and positive integer $d$ and $\ell\le d$, there are $C_\ell, C'_\ell, \gamma_\ell > 0$ such that the following holds. Let $\Gamma$ be a $(c,t)$-sparse graph with $n$ vertices and $G \subseteq \Gamma$ be a bipartite subgraph with parts $U,R$ with $|U|=m,|R|=r$ and edge density $q \ge C_\ell m^{- 1/d} t^{1/d}$. Assume that $r\ge C'_\ell t(m/t)^{\ell/d}$. Let $H$ be $d$-bounded with vertex parts $A$ and $B$, $\ell$ vertices in $A$ complete to $B$, $|A|=\ell+a$ and $|B|=b$. Then, there are at least $\gamma_\ell m^{a+\ell} r^b q^{e(H)}$ copies of $H$ in $G$ which are induced in $\Gamma$. 
\end{lemma}
\begin{proof}
    We prove the statement by induction on $\ell$. We let $\beta>0$ be sufficiently small (depending on $H$, $c$ and $d$) and let $\varepsilon>0$ be given by Lemma \ref{lemma:drc_count_1}. For $\ell\ge 1$, we choose the constants $C_\ell$ and $C'_\ell$ as follow.
    \begin{alignat*}{2}
        &C_\ell = 1/(2(1-\beta)\eps\delta)^{\ell}, \,\,\,\, C_\ell' = \delta^{-2\ell}, \,\,\,\, \gamma_\ell = \frac{\gamma \delta^{e(H)}}{2}((1-\beta)^{e(H)}\eps^{e(H)}\delta^{2b}/2)^{\ell-1},
    \end{alignat*}
    where $\gamma = \gamma(\delta,c,\beta,\eps,H)$ is defined in Claim \ref{claim:copies_T}.

    {\bf Base case:} $\ell = 1$. 
    
    Apply Lemma \ref{lem:reduce} to obtain a $\delta^2$-reduced induced subgraph $G' = G[U', R]$ with $|U'|=m'$. In particular, for $x \ge 1 + 1/v(H)$ we have
    $$
        m' \cdot d_\mathrm{avg}(U')^x \ge \delta^x m \cdot d_\mathrm{avg}(U)^x.
    $$
    Dividing each side by $r^x$, we get
    \begin{equation} \label{eq:reduced}
        m' (q')^x \ge \delta^x m q^x,
    \end{equation}
    where $q' = e(G') / (m' r)$ is the density of $G'$. By Lemma \ref{lemma:drc_count_1} and Claim \ref{claim:copies_T}, the number of copies of $H$ in $G'$ induced in $\Gamma$ is at least 
    \[
    \gamma \frac{m'}{2} (m')^ar^b (q')^{e(H)} = \frac{\gamma}{2} r^b \left(m' \cdot (q')^{e(H)/(a+1)} \right)^{a+1} \ge \frac{\gamma \delta^{e(H)}}{2} r^b \left(m q^{e(H) / (a+1)}\right)^{a+1} = \gamma_1 m^{a+1} r^b q^{e(H)}.
    \]
    The first inequality follows from \eqref{eq:reduced} applied with $x = e(H) / (a+1) \ge  1 + 1 / v(H)$.
    
    {\bf Inductive step:} Assume the conclusion of the lemma holds when $H$ has $\ell-1$ vertices complete to the other side. 
    
    Apply Lemma \ref{lem:reduce} to obtain a $\delta^2$-reduced induced subgraph $G' = G[U', R]$ with $|U'|=m'$. By Lemma \ref{lemma:drc_count_1}, at least $m'/2$ vertices $u\in U'$ are such that $u$ is $(\beta,\eps,k)$-good for all $k\le d$. In particular, for at least $(1-\beta)|N_{G'}(u)|$ vertices $x\in N_{G'}(u)$, we have
    $$
        |N_{G'}(x)\setminus N_\Gamma(u)| \ge \eps m' q',
    $$
    where $q' = e(G') / (m'r)$.

    Let $H'$ be the graph obtained from $H$ by removing one of $\ell$ vertices in $A$ complete to $B$. Let $G''$ be the induced subgraph of $G$ between $N_{G'}(u)$ and $U' \setminus N_\Gamma(u)$. We then have 
    \begin{align*}
        e(G'') &= \sum_{x\in N_{G'}(u)}|N_{G'}(x)\setminus N_\Gamma(u)| \ge (1-\beta) |N_{G'}(u)| \cdot \eps m' q',
    \end{align*}
    thus 
    \begin{equation} \label{eq:q_density}
        q'' \ge (1 - \beta) \eps q'. 
    \end{equation}
    Towards obtaining more precise estimate on the density $q''$ of $G''$, we have
    \begin{align*}
        \frac{e(G'')}{|N_{G'}(u)|(m')^{1-1/d}} &\ge (1-\beta) \eps \frac{m' q'}{(m')^{1-1/d}} \\
        &\ge (1-\beta)\eps \frac{d_\mathrm{avg}(U')(m')^{1/d}}{r} \\
        &= (1-\beta)\eps \frac{d_\mathrm{avg}(U)m^{1/d}}{r} \left(\frac{d_\mathrm{avg}(U')^d m'}{d_\mathrm{avg}(U)^d m}\right)^{1/d} \\
        &\ge (1-\beta)\eps\delta \frac{d_\mathrm{avg}(U)m^{1/d}}{r} = (1 - \beta) \eps \delta q m^{1/d} \ge C_\ell (1 - \beta) \eps \delta t^{1/d},
    \end{align*}
    where the third inequality follows from the property of Lemma \ref{lem:reduce} applied with $x = d$. Hence, the edge density $q''$ of $G''$ satisfies $q'' \ge C_\ell (1-\beta)\eps\delta (m')^{-1/d} t^{1/d}$. Our choice of $C_\ell$ guarantees that $C_\ell (1-\beta)\eps \delta \ge C_{\ell-1}$. Moreover, 
    \[
    |N_{G'}(u)| \ge \delta d_\mathrm{avg}(U') \ge \delta^{2} d_\mathrm{avg}(U) = \delta^{2} r q \ge \delta^{2}  \cdot C'_\ell t(m/t)^{\ell/d} \cdot C_\ell (t/m)^{1/d},
    \]
    which is at least $C'_{\ell-1}t(m/t)^{(\ell-1)/d}$ for our choice of $C'_\ell$. 
    Then, by the inductive hypothesis, we have that the number of copies of $H'$ in $G''$ is at least 
    \begin{align*}
        &\gamma_{\ell-1} (m')^{a+\ell-1} |N_{G'}(u)|^{b} (q'')^{e(H)-b} \\
        &\stackrel{\eqref{eq:q_density}}{\ge} \gamma_{\ell-1} (1 - \beta)^{e(H)} \eps^{e(H)} \cdot (m')^{a+\ell-1} \cdot |N_{G'}(u)|^{b} \cdot (q')^{e-b} \\
        &\ge \gamma_{\ell-1} (1 - \beta)^{e(H)} \eps^{e(H)} \cdot (m')^{a+\ell-1} \cdot \delta^{2b} d_\mathrm{avg}(U')^b  \cdot (d_\mathrm{avg}(U')/r)^{e(H)-b}\\
        &\ge \gamma' \cdot m^{a+\ell-1} \cdot d_\mathrm{avg}(U)^{e(H)} r^{b-e(H)} \\
        &= \gamma' m^{a+\ell-1}r^bq^{e(H)},
    \end{align*}
    where $\gamma' = \gamma_{\ell-1} (1 - \beta)^{e(H)} \eps^{e(H)} \delta^{2b}$. Summing over vertices $u\in U'$ such that $N_{G'}(u)$ is $(\beta,\eps,k)$-good for all $k\le d$, we obtain that the number of copies of $H$ in $G$ which are induced in $\Gamma$ is at least 
    $$
        (m/2) \gamma' m^{a+\ell-1}r^bq^{e(H)} \ge \gamma_\ell  m^{a+\ell}r^bq^{e(H)}.
    $$
    This completes the inductive step. 
\end{proof}

\section{Trees}\label{sec:tree}

The proof of Theorem \ref{thm:tree} is based on a delicate inductive argument which `glues' induced copies of subtrees of $T$. To this end, we start by introducing necessary notions. We only consider labelled copies of $T$, which we identify with an injective function $\phi \colon V(T) \rightarrow V(G)$. For a given copy $\phi$ of $T$ and a leaf $v \in V(T)$, we denote by $\phi - v$ a copy of $T - v$ given by $\phi\restriction_{V(T) - v}$.

\noindent {\bf Unique neighbors.} Given graphs $G$ and $\Gamma$ on the same vertex set, we say that a subset of vertices $S \subset V(G)$ is \emph{$c$-unique} if for every $v \in S$ we have
$$
    |N_G(v) \setminus N^+_\Gamma(S - v)| \ge c^{|S|-1} |N_G(v)|.
$$
We say that a copy $\phi$ of $T$ is $c$-unique if the set $\phi(T)$ is $c$-unique.

\noindent {\bf Extending $T - v$ to $T$.} Let $T$ be a tree with at least two vertices, $v \in V(T)$ a leaf in $T$ and $h_v \in V(T)$ the unique neighbor of $v$. Suppose $\lambda_{T - v}$ and $\lambda_T$ are probability measures over copies of $T - v$ and $T$, respectively, in some graph $G$. We say that $\lambda_T$ \emph{$(K,\eps, \kappa)$-extends} $\lambda_{T - v}$, for $K, \eps, \kappa > 0$, if the following two properties hold:
 \begin{equation} \label{prop:upper}
    \lambda_{T}(\phi) \le  \frac{K \lambda_{T - v}(\phi - v)}{ \deg_G(\phi(h_v))} \quad \text{for every $\phi \in \mathrm{supp}(\lambda_{T})$}
\end{equation}
and
\begin{equation} \label{prop:lower}
    \Pr_{\phi' \sim \lambda_{T - v}} \left[ \lambda_{T}(\phi') \ge \eps  \lambda_{T - v}(\phi') \right] \ge 1 - \kappa,
\end{equation}
where
$$
\lambda_T(\phi') = \lambda_T(\{ \phi \in \mathrm{supp}(\lambda_T) \colon \phi - v = \phi'\}).
$$
Note that \eqref{prop:upper} implies
\begin{equation} \label{prop:upper_phi-v}
    \lambda_T(\phi') \le K \lambda_{T - v}(\phi') \quad \text{ for every } \phi' \in \mathrm{supp}(\lambda_{T - v}).
\end{equation}

The following two lemmas are indirectly related to these properties, thus we state them now.

\begin{lemma} \label{lemma:good_vertices}
    Let $\Gamma$ be a $(c, t)$-sparse graph. Suppose $\pi_1$ and $\pi_2$ are probability measures over $V(\Gamma)$ such that $\pi_i(v) < \gamma / t$ for all $v \in V(\Gamma)$ and $i \in \{1, 2\}$, for some $0 < \gamma < 1$. Define the \emph{good} set of vertices with respect to $\pi_2$ as
    $$
        D(\pi_2) = \{x \in V(\Gamma) \colon \pi_2(\{ y \in V(\Gamma) \colon (x, y) \not \in \Gamma\}) \ge c\}.
    $$
    Then
    $$
        \Pr_{x \sim \pi_1}\left[ x \in D(\pi_2) \right] > 1 - \gamma.
    $$
\end{lemma}
\begin{proof}
	Assume for the sake of contradiction that $\Pr_{x \sim \pi_1}\left[ x \in D(\pi_2) \right] \le 1 - \gamma$. Since $\pi_1(v) < \gamma/t$ for all $v$, there is a set $B_1\subseteq \Gamma$ of size at least $t$ for which $\pi_2(\{y\in V(\Gamma): (x,y)\notin \Gamma\}) < c$, and hence $\pi_2(\{y\in V(\Gamma): (x,y)\in \Gamma\}) > 1-c$, for all $x\in B_1$. Thus, letting $U_1$ denote the uniform distribution on $B_1$, 
\[
	\mathbb{E}_{x_1\sim U_1, x_2\sim \pi_2} [\mathbb{I}((x,y)\in \Gamma)] > 1-c.
\]

    Recall the standard fact that, any distribution $P$ on a finite set $V$ with $\max_{v\in V} P(v) \le \theta$ can be written as a convex combination of distributions, each of which is uniform over its support (which is a subset of $V$) of size $\lfloor 1/\theta\rfloor$ (see \cite[Lemma 6.10]{V}). Applying this fact to $\pi_2$ with $\theta = \gamma/t$, and by averaging, we obtain that there is a subset $B_2\subseteq \Gamma$ of size $\lfloor t/\gamma\rfloor \ge t$ such that, for $U_2$ the uniform distribution on $B_2$, 
\[
	\mathbb{E}_{x_1\sim U_1, x_2\sim U_2} [\mathbb{I}((x,y)\in \Gamma)] > 1-c.
\]
However, this contradicts the $(c,t)$-sparse assumption on $\Gamma$. 	
\end{proof}

\begin{lemma} \label{lemma:good_vertices_2}
    Let $\Gamma$ be a $(c, t)$-sparse graph. Suppose $\pi_1$ and $\pi_2$ are probability measures over $V(\Gamma)$ such that $\pi_i(v) < \gamma / t$ for all $v \in V(\Gamma)$ and $i \in \{1, 2\}$, for some $0 < \gamma < 1$. Moreover, let $\mathcal{S} = \{S_v \colon v \in \mathrm{supp}(\pi_2) \}$ be a family of subsets of $V(\Gamma)$, each of size $|S_v| \ge t$. Denote with $D(\pi_2, \mathcal{S})$ the \emph{good} set of vertices with respect to $\pi_2$ and the family $\mathcal{S}$, defined as
    $$
        D(\pi_2, \mathcal{S}) = \{x \in V(\Gamma) \colon \pi_2(\{ y \in V(\Gamma) \colon |S_y \setminus N_\Gamma(x)| < c |S_y| \}) < \sqrt{\gamma} \}.
    $$
    Then
    $$
        \Pr_{x \sim \pi_1}\left[ x \in D(\pi_2, \mathcal{S}) \right] > 1 - \sqrt{\gamma}.
    $$
\end{lemma}
\begin{proof}
    For $x, y \in V(\Gamma)$, write $x \triangleright y$ as a shorthand for $|S_y \setminus N_\Gamma(x)| < c |S_y|$. As $\Gamma$ is $(c, t)$-sparse and $|S_y| \ge t$, there are at most $t$ vertices $x$ such that $x \triangleright y$, for any $y \in V(\Gamma)$. Therefore, $\Pr_{x,y}[x \triangleright y] < \gamma$ with $x \sim \pi_1$ and $y \sim \pi_2$ sampled independently. Let $B = V(G) \setminus D(\pi_2, \mathcal{S}) = \{x\in V(\Gamma): \pi_2(\{y\in V(\Gamma): x \triangleright y\})\ge \sqrt{\gamma}\}$. By Markov's inequality, 
    $$ 
        \Pr_{x\sim \pi_1}[x\in B] \le \frac{\Pr_{x,y}[x\triangleright y]}{\sqrt{\gamma}} < \sqrt{\gamma}.
    $$%
    Thus, 
    $$
        \Pr_{x \sim \pi_1}\left[ x \in D(\pi_2, \mathcal{S}) \right] > 1 - \sqrt{\gamma}.
    $$
\end{proof}

\noindent {\bf Peeling a tree.} Let $T$ be a (labelled) tree, and let $\mathcal{S}(T)$ denote the family of all subtrees of $T$. A mapping $\nu \colon \mathcal{S}(T) \rightarrow 2^{V(T)}$ is \emph{$T$-peeling} if the following holds: (i) $\nu(\{v\}) = \{v\}$ for $v \in V(T)$, and (ii) for $|T'| \ge 2$ we have $\nu(T') = \{u, w\}$ where $u$ and $w$ are leaves in $T'$ such that $u \in \nu(T' - w)$ and $w \in \nu(T' - u)$. 

\begin{lemma} \label{lemma:peeling}
    For every tree $T$ there exists a $T$-peeling mapping.
\end{lemma}
\begin{proof}
    We prove the lemma by induction on the size of $T$. For $|T| \le 2$, the claim trivially holds. Consider some $|T| \ge 3$, and pick a leaf $v \in T$. Set $T_v = T - v$, and let $\nu_v$ be a $T_v$-peeling mapping. For a subtree $T' \subseteq T$ which does not contain $v$ we set $\nu(T') = \nu_v(T')$. For a subtree $T' \subseteq T$ which contains $v$, take $w \in \nu_v(T' - v)$ which is not adjacent to $v$ and set $\nu(T) = \{v, w\}$.
\end{proof}

We are now ready to prove Theorem \ref{thm:tree}.

\begin{proof}[Proof of Theorem \ref{thm:tree}]
    By sequentially removing vertices of degree less than $Ct$ in $G$, we eventually end up with a graph with $n' > 1$ vertices, at least $Ctn'$ edges, and minimum degree at least $Ct$. Therefore, without loss of generality we may assume that $G$ has minimum degree at least $Ct$.

    Consider a $T$-peeling mapping $\nu$. We show that for each tree $T' \subseteq T$ there exists a probability measure $\lambda_{T'}$ over $c$-unique copies of $T'$ in $G$ which are induced in $\Gamma$, such that $\lambda_{T'}$ $(K_\ell, \eps_\ell, \kappa_\ell)$-extends $\lambda_{T' - v}$ for each $T' \subseteq T$ with at least two vertices and $v \in \nu(T')$, where $\ell = |V(T')|$. As the parameters depend only on the size of a tree, we simply write $\lambda_{T'}$ extends $\lambda_{T' - v}$. Here $K_i > 1$ and $0 < \eps_i, \kappa_i < 1$, for $2 \le i \le |V(T)|$, are defined as follows:
    \begin{alignat*}{2}
        &K_2 = 2, \eps_2 = 1/3, \kappa_{|V(T)|} = 1/2 \\
        &\eps_i = c / (2 K_{i-1}), K_i = 2 K_{i-1} / (\eps_{i-1}^2 c) && \quad \text{ for } i \ge 3 \\
        &\kappa_i = \kappa_{i+1} / (2 K_i) && \quad \text{ for } 2 \le i < |V(T)|,
    \end{alignat*}
    and $C$ is chosen to be sufficiently large. We prove the existence of such measures by induction on the size of $T'$, handling the base cases $T'$ has one or two vertices first.
    
    \medskip
    \textbf{$T'$ is a single vertex.} Let $T' = \{v\}$. We define $\lambda_{v}$ as
    $$
        \lambda_{v}(x) = \frac{\deg_G(x)}{2 e(G)}.
    $$

    \medskip
    \textbf{$T'$ is an edge.} Let $T'$ be an edge $\{u, w\}$, and note that necessarily $\nu(T') = \{u, w\}$ (recall that $T$ is a labelled tree, so we distinguish between $u$ and $w$). Consider the following process: Sample $x_u \sim \lambda_{u}$, and then sample $x_w \in N_G(x_u)$ uniformly at random from $N_G(x_u)$. If $|N_G(x_u) \setminus N_\Gamma(x_w)| \ge c |N_G(x_u)|$ and $|N_G(x_w) \setminus N_\Gamma(x_u)| \ge c |N_G(x_w)|$, output a copy $\phi = \{u \mapsto x_u, w \mapsto x_w\}$ of $T'$ and otherwise reject. When we output, $\phi$ is a $c$-unique copy of $T'$, and as $T'$ contains a single edge it is trivially induced in $\Gamma$. Define $\lambda_{T'}$ to be the corresponding output distribution. We will make use of the fact that first sampling $x_w \sim \lambda_w$ and then $x_u \in N_G(x_w)$ yields the same probability distribution. Indeed, the distribution can alternatively be described as sampling a uniformly random edge $\{x_u,x_w\}$ in $G$, then accept if $|N_G(x_u) \setminus N_\Gamma(x_w)| \ge c |N_G(x_u)|$ and $|N_G(x_w) \setminus N_\Gamma(x_u)| \ge c |N_G(x_w)|$. 

    For each $x \in V(\Gamma)$, as $\Gamma$ is $(c, t)$-sparse and $|N_G(x)| \ge Ct$, we have $|N_G(x) \setminus N_\Gamma(y)| \ge c |N_G(x)|$ for all but at most $t$ vertices $y \in N_G(x)$. As we can sample $T'$ starting with either $x_u$ or $x_w$, the above claim (with $x=x_u$ or $x=x_w$)  implies that $\phi$ is $c$-unique with probability at least $1 - 2 t / (Ct) > 1 - \kappa_2/2$. In other words, the probability $p_{\mathrm{succ}}$ of successfully outputting a copy of $T'$ is at least $1 - \kappa_2/2 > 1/2$. We now verify that $\lambda_{T'}$ $(K_2, \eps_2, \kappa_2)$-extends $\lambda_{T' - u} = \lambda_w$ and $\lambda_{T' - w} = \lambda_u$.
    
    Property \eqref{prop:upper} states that for every (ordered) edge $e = (x_u,x_w) \in \mathrm{supp}(\lambda_{T'})$ we have
    $$
        \lambda_{T'}(e) = \frac{1}{p_\mathrm{succ}} \cdot \frac{\lambda_{u}(x_u) }{\deg_G(x_u)} \le \frac{2 \lambda_u(x_u)}{\deg_G(x_u)},
    $$
    which follows from the lower bound on the success probability and the definition of the process. The analogous statement holds for $w$ in place of $u$.

    Property \eqref{prop:lower} states that
    \begin{equation} \label{eq:x_u_lower}
        \Pr_{x_u \sim \lambda_u} \left[ \lambda_{T'}(u \mapsto x_u) \ge \lambda_u(x_u) / 3 \right] \ge 1 - \kappa_2,
    \end{equation}
    where $\lambda_{T'}(u \mapsto x_u) = \lambda_{T'}(\{(x_u, x_w) \in \mathrm{supp}(\lambda_{T'})\})$. Due to symmetry, this implies the analogous statement for $\lambda_w$. Let $B \subseteq V(G)$ denote the set of vertices such that for each $x \in B$ there are less than $\deg_G(x)/2$ vertices $y \in N_G(x)$ for which $|N_G(y) \setminus N_\Gamma(x)| \ge c |N_G(y)|$. For each $x \not \in B$, there are at least $\deg_G(x)/2 - t > \deg_G(x)/3$ vertices $y \in N_G(x)$ such that $(x,y)$ is $c$-unique. Therefore, for $x_u \not \in B$ we have
    $$
        \lambda_{T'}(u \mapsto x_u) \ge \sum_{\substack{y \in N_G(x_u) \\ (x,y) \text{ $c$-unique}}} \lambda_{u}(x_u) / \deg_G(x_u) \ge \lambda_u(x_u) / 3.
    $$
    To prove \eqref{eq:x_u_lower}, it suffices to show $\lambda_u(B) \le \kappa_2$. Suppose otherwise, then
    \begin{align*}
        1 - p_{\mathrm{succ}} &= \Pr_{x_u, x_w}\left[ (x_u, x_w) \text{ not } c\text{-unique} \right] \\
        & \ge \Pr_{x_u \sim \lambda_u}(x_u \in B)  \Pr_{x_u, x_w}\left[ (x_u, x_w) \text{ not } c\text{-unique} \mid x_u \in B \right] > \kappa_2 / 2,
    \end{align*}
    where the joint probability $(x_u, x_w)$ is as given by the process. This is a contradiction with the previously established lower bound on $p_{\mathrm{succ}}$.

    \medskip
    \textbf{$T'$ has at least three vertices.} Suppose now $T'$ has $\ell \ge 3$ vertices and the statement holds for every proper subtree $T'' \subset T'$. Let $\nu(T') = \{u, w\}$, and let $h_u$ and $h_w$ be unique neighbors of $u$ and $w$ in $T'$. Note that $h_u$ and $h_w$ could be the same vertex. Set $R = T' - \nu(T)$, $R_u = T' - w$ and $R_w = T' - u$. By the inductive assumption and the definition of a $T$-peeling mapping, $\lambda_{R_u}$ and $\lambda_{R_w}$ extend $\lambda_R$.
    
    Consider the following process:
    \begin{enumerate}[(I)]
        \item \label{step:R} Sample $\rho \sim \lambda_R$. If $\lambda_{R_u}(\rho) < \eps_{\ell-1} \lambda_R(\rho)$ or $\lambda_{R_w}(\rho) < \eps_{\ell-1} \lambda_R(\rho)$, reject.
        
        \item \label{step:x} Sample $\rho_u \sim \lambda_{R_u} \mid \rho$, that is, $\rho_u$ is sampled according to $\lambda_{R_u}$ conditioned on agreeing with $\rho$ on $R$. Similarly, independently sample $\rho_w \sim \lambda_{R_w} \mid \rho$. Let $x_u = \rho_u(u)$ and $x_w = \rho_w(w)$, and define $\phi$ to be a mapping given by $\rho \cup \{u \mapsto x_u, w \mapsto x_w\}$. Reject if any of the following events happen: (i) $x_u = x_w$, (ii) $\phi$ is not $c$-unique, or
        (iii) $x_u$ and $x_w$ are adjacent in $\Gamma$.
        
        \item Output the copy $\phi$ of $T'$.
    \end{enumerate}    
    Note that if we did not reject in step \ref{step:R} then $\lambda_{R_v}(\rho) > 0$ implies $\lambda_{R_v} \mid \rho$ is defined, for $v \in \{u, w\}$. Events (i)--(iii) ensure that $\phi$ is a $c$-unique copy of $T'$ in $G$ which is induced in $\Gamma$, thus we define $\lambda_{T'}$ to be the corresponding output distribution. It remains to show that the probability $p_{\mathrm{succ}}$ of outputting a copy of $T'$ in step (3) is non-zero and $\lambda_{T'}$ extends $\lambda_{T' - u} = \lambda_{R_w}$ and $\lambda_{T' - w} = \lambda_{R_u}$. 
    
    For brevity, throughout the proof we denote with $\pi_{v, \rho} = \lambda_{R_v} \mid \rho$ the probability measure corresponding to $x_v$ with respect to $\rho$, for $v \in \{u, w\}$.

    \begin{claim} \label{claim:pi}
        Suppose $\lambda_{R_u}(\rho) \ge \eps_{\ell-1} \lambda_R(\rho)$. Then for every $x \in N_G(\rho(h_u))$ we have 
        $$
            \pi_{u, \rho}(x) \le \frac{L}{\deg_G(\rho(h_u))} \le \frac{L}{C t}.
        $$
        where $L = K_{\ell-1} \eps_{\ell-1}^{-1}$. The analogous statement holds for replacing $u$ in each instance by $w$.
    \end{claim}
    \begin{proof}
        Consider some $x \in N_G(\rho(h_u))$ (otherwise $\pi_{u, \rho}(x) = 0$) and let $\rho_x = \rho \cup \{u \mapsto x\}$. Then, as $\lambda_{R_u}$ extends $\lambda_{R}$, we have
        $$
            \pi_{u, \rho}(x) = \frac{\lambda_{R_u}(\rho_x)}{\lambda_{R_u}(\rho)} \stackrel{\eqref{prop:upper}}{\le} \frac{ K_{\ell-1} \lambda_{R}(\rho) }{\deg_G(\rho(h_u)) } \cdot \frac{1}{\lambda_{R_u}(\rho)},
        $$
        thus the claim follows from $\lambda_{R_u}(\rho) \ge \eps_{\ell-1} \lambda_{R}(\rho)$ and the minimum degree assumption.
    \end{proof}
    
    \noindent {\bf Success probability $p_{\mathrm{succ}}$.} By \eqref{prop:lower}, the probability of rejecting in step \ref{step:R} is at most $2 \kappa_{\ell-1}$. Let us condition on $\rho$ not being rejected, that is, $\lambda_{R_u}(\rho) \ge \eps_{\ell-1} \lambda_R(\rho)$ and $\lambda_{R_w}(\rho) \ge \eps_{\ell-1} \lambda_R(\rho)$. By Claim \ref{claim:pi} we have $x_u = x_w$ with probability at most $L / C$. Next, as $\rho_u$ is a $c$-unique copy of $R_u$ and $\Gamma$ is $(c, t)$-sparse, for all but at most $(\ell-1)t$ choices of $x_w \in N_G(\rho(h_w))$ we have
    $$
        |N_G(\rho_u(v)) \setminus \left( N_\Gamma^+(\rho_u(R_u - v)) \cup N_\Gamma(x_w) \right)| \ge c^{\ell-1} |N_G(\rho_u(v))| \quad \text{for every } v \in R_u.
    $$
    We implicitly assume here $C > c^{-\ell}$, so that the right hand side of the previous inequality is at least $t$. Analogous statement holds for $x_u \in N_G(\rho(h_u))$. By Claim \ref{claim:pi} and the fact that $x_u$ and $x_w$ are sampled independently, we conclude that $\phi$ is not $c$-unique with probability at most $2 (\ell-1) L / C$. Finally, by Lemma \ref{lemma:good_vertices} applied with $\pi_{u, \rho}$ (as $\pi_1$) and $\pi_{w, \rho}$ (as $\pi_2$), $x_u$ and $x_w$ are not adjacent in $\Gamma$ with probability at least $(1 - L/C) c$. All together, for sufficiently large $C$ we have $p_{\mathrm{succ}} > c/2$.

    \medskip
    Next, we verify that $\lambda_{T'}$ $(K_{\ell}, \eps_{\ell}, \kappa_{\ell})$-extends $\lambda_{R_u}$ and $\lambda_{R_w}$. The two cases are identical, thus we only do it for $\lambda_{R_u}$.

    \noindent {\bf Property \eqref{prop:upper}.} Consider some $\phi \in \mathrm{supp}(\lambda_{T'})$, and let $x_u = \phi(u)$ and $x_w = \phi(w)$. Then $\rho = \phi - \{u, w\}$ is not rejected in step \ref{step:R}, as otherwise $\phi$ would not be produced by the procedure. We have
    \begin{align*}
        \lambda_{T'}(\phi) &= \frac{1}{p_{\mathrm{succ}}} \cdot \lambda_{R}(\rho) \pi_{u, \rho}(x_u) \pi_{w, \rho}(x_w) \\
        &\le \frac{2}{c} \cdot \lambda_R(\rho) \frac{\lambda_{R_u}(\phi - w)}{\lambda_{R_u}(\rho)} \pi_{w, \rho}(x_w) \le \frac{2}{c \eps_{\ell-1}} \lambda_{R_u}(\phi - w) \pi_{w, \rho}(x_w),
    \end{align*}
    where the last inequality follows from $\lambda_{R_u}(\rho) \ge \eps_{\ell-1} \lambda_R(\rho)$. Claim \ref{claim:pi} applied on $\pi_{w, \rho}(x_w)$ gives \eqref{prop:upper}.
    
    \noindent {\bf Property \eqref{prop:lower}.} Consider a copy $\rho_u \in \mathrm{supp}(\lambda_{R_u})$ of $R_u$. Set $x_u = \rho_u(u)$ and let $\rho = \phi_u - u$ be a corresponding copy of $R$. For $x \in \mathrm{supp}(\pi_{w, \rho})$ define $S_x = N_G(x) \setminus N^+_\Gamma(\rho(R))$, and denote the family of all such subsets by $\mathcal{S}$. By the fact that $\rho \cup \{w \mapsto x\}$ is $c$-unique for every $x \in \mathrm{supp}(\pi_{w, \rho})$, we have $|S_x| \ge t$. Suppose 
    \begin{equation} \label{eq:good_phi_u}
        \lambda_{R_u}(\rho) \ge \eps_{\ell - 1} \lambda_R(\rho) \quad \text{ and } \quad x_u \in D(\pi_{w, \rho}) \cap D(\pi_{w, \rho,}, \mathcal{S}),
    \end{equation}
    where $D(\pi_{w, \rho})$ is as defined in Lemma \ref{lemma:good_vertices}, and $D(\pi_{w, \rho}, \mathcal{S})$ as defined in Lemma \ref{lemma:good_vertices_2}. Then, along similar lines as when estimating $p_{\mathrm{succ}}$, we have
    \begin{equation} \label{eq:rho_uw_supp}
        \Pr_{x_w \sim \pi_{w, \rho}}\left[ \rho_u \cup \{w \mapsto x_w\} \in \mathrm{supp}(\lambda_{T'}) \right] \ge c - \frac{\ell t L}{Ct} > c / 2,
    \end{equation}
    for $C$ sufficiently large. Briefly, $c$ is the bound on the probability that $x_w$ and $x_u$ are not adjacent in $\Gamma$ (coming from $x_u \in D(\pi_{w,\rho})$), and the remaining term takes into account the probability that $x_u$ and $x_w$ are distinct and $\phi_u \cup \{w \mapsto x_w\}$ is $c$-unique (recall that $\rho_u$ is $c$-unique by the definition). Therefore,
    \begin{align*}
        \lambda_{T'}(\rho_u) &\ge \lambda_{R}(\rho) \pi_{u, \rho}(x_u) \pi_{w, \rho}(\{ x_w \colon \rho_u \cup \{w \mapsto x_w\} \in \mathrm{supp}(\lambda_{T'}) \}) \\
        &\stackrel{\eqref{eq:rho_uw_supp}}{\ge} \lambda_R(\rho) \frac{\lambda_{R_u}(\rho_u)}{\lambda_{R_u}(\rho)} \cdot c/2 \stackrel{\eqref{prop:upper_phi-v}}{\ge} \frac{c}{2 K_{\ell-1}} \lambda_{R_u}(\rho_u).
    \end{align*}
    In the last inequality we used the assumption that $\lambda_{R_u}$ extends $\lambda_R$. To summarise, if $\rho_u$ satisfies \eqref{eq:good_phi_u} then $\lambda_{T'}(\rho_u) \ge \eps_{\ell} \lambda_{R_u}(\rho_u)$ for $\eps_\ell = c / (2 K_{\ell-1})$. We estimate the probability that $\rho_u$ does not satisfy \eqref{eq:good_phi_u}:
    \begin{multline*}
        \Pr_{\rho_u \sim \lambda_{R_u}}\left[ \lambda_{R_u}(\rho) < \eps_{\ell - 1} \lambda_R(\rho) \lor x_u \not \in D(\pi_{w, \rho}) \cap D(\pi_{w, \rho,}, \mathcal{S}) \right] = \\
        \Pr_{\rho_u \sim \lambda_{R_u}}\left[ \lambda_{R_u}(\rho) < \eps_{\ell-1} \lambda_{R}(\rho) \right] + \Pr_{\rho_u \sim \lambda_{R_u}}\left[ x \not \in D(\pi_{w, \rho}) \cap D(\pi_{w, \rho,}, \mathcal{S}) \mid \lambda_{R_u}(\rho) \ge \eps_{\ell-1} \lambda_{R}(\rho) \right].
    \end{multline*}
    By \eqref{prop:lower} and \eqref{prop:upper_phi-v}, the first term is at most $\kappa_{\ell-1} K_{\ell-1}$. The second term is at most $L / C + \sqrt{L / C}$ by Lemma \ref{lemma:good_vertices} and Lemma \ref{lemma:good_vertices_2}. For $C$ sufficiently large and $\kappa_{\ell-1}$ sufficiently small with respect to $K_{\ell-1}$, we get a desired upper bound on $\rho \sim \lambda_{R_u}$ failing \eqref{prop:lower}. This finishes the proof.
\end{proof}

\section{Applications of Theorem \ref{thm:drc} for certain hereditary families} \label{sec:string}

In this section, we prove Theorem \ref{string} which states that $\textrm{ex}(\Gamma,\mathcal{S})=O_c(tn)$ for $\Gamma$ a $(c,t)$-sparse graph and $\mathcal{S}$ the family of string graphs. The following lemma from \cite{FPS} was proved using an iterative application of  Lee's separator lemma for string graphs. 

\begin{lemma}\label{denseinduced} (Lemma 3, page 224 of \cite{FPS})
For each $\varepsilon>0$ there is $C=C(\varepsilon)$ such that 
every string graph which has average degree $d$ has an induced subgraph on at most 
$Cd$ vertices which has average degree at least $(1-\varepsilon)d$. 
\end{lemma}

Lemma \ref{denseinduced} shows that string graphs have dense induced subgraphs with average degree close to that of the original graph. Together with Corollary \ref{cor:deg_d}, this implies Theorem \ref{string}.

\begin{proof}[Proof of Theorem \ref{string}]
Let $\Gamma$ be a $(c,t)$-sparse graph on $n$ vertices and $G \subseteq \Gamma$ be a subgraph which is a string graph with the maximum number of edges. Let $d$ be the average degree of $G$. We may suppose suppose $d \geq At$ for some large constant $A$ dependent on $c$ as otherwise we are done. Applying Lemma \ref{denseinduced} with $\varepsilon=1/2$, there is an induced subgraph $G'$ of $G$ with average degree at least $d/2$ and with at most $Cd$ vertices. As $G'$ is a subgraph of $\Gamma$, $G'$ is $(c,t)$-sparse. As $G'$ is an induced subgraph of $G$, $G'$ is a string graph. From Corollary \ref{cor:deg_d} applied to $G'$ as $\Gamma$ (noting that $G'$ is a string graph and hence does not contain any induced copy of the bipartite graph $B$ which is the $1$-subdivision of $K_5$), we get that $G'$ has small edge density, contradicting that its edge density is at least $(d/2)/(Cd)=1/(2C)$. 
\end{proof}

As an application of Theorem \ref{string}, we determine the extremal number within a constant factor when $\mathcal{P}$ is the property of being an incomparability graph of a partially ordered set, in which case for $\Gamma=G(n,p)$ with $0<p<1$ fixed, with high probability we have $\mathrm{ex}(\Gamma,\mathcal{P})=\Theta(n\log n)$. Indeed, by \cite{GRU, L, PT}, any incomparability graph is a string graph. Thus, $\mathrm{ex}(\Gamma,\mathcal{P})=O(n\log n)$ by Theorem \ref{string}. On the other hand, we saw that $\Gamma$ contains a subgraph $G$ with $\Theta(n\log n)$ edges, which is a disjoint union of cliques of size $\Theta(\log n)$. Such a graph $G$ is an incomparability graph of a partially ordered set. The same argument shows that if $\Gamma = P_q$ is the Paley graph for a square $q$, then $\ex(\Gamma, \mathcal{P}) = \Theta(q^{3/2})$. 

In contrast, if $\mathcal{P}$ is the family of comparability graphs (complements of incomparability graphs) or perfect graphs, we have with high probability, for $\Gamma = G(n,p)$, $\mathrm{ex}(\Gamma,\mathcal{P})$ is asymptotically $pn^2/4$ by the theorem of Alon, Krivelevich and Samotij \cite{AKS23}, as these families contain all bipartite graphs, but there are graphs of chromatic number $3$ like the five-cycle which are not in the family.

\section{Concluding remarks} \label{sec:concluding}

{\bf Improved bound for Paley graphs.} It is conjectured that the clique number of Paley graphs with a prime number $q$ of vertices is much smaller than the upper bound $\sqrt{q}$. This bound has only been improved recently by a constant factor \cite{DSW,HP}. An improvement on the bound in Theorem \ref{paley} when $q$ is not a square would imply the clique number of the Paley graph is $o(\sqrt{q})$. Furthermore, the Paley Graph Conjecture implies that such graphs should be $(c,q^{\varepsilon})$-sparse for any fixed $c<1/2$ and $\varepsilon>0$ and $q$ sufficiently large (see \cite{Yip} for more). Together with our results, this would imply for $n$ prime and $\Gamma$ a Paley graph on $n$ vertices that 
$\mathrm{ex}(\Gamma,\mathcal{P}_{K_{s,r}})=n^{2-1/s+o(1)}$ for fixed $r$ and $s$ with $r$ sufficiently large in $s$ and $n \to \infty$.

\medskip
{\bf Even cycles.} Corollary \ref{lowercor-fam} shows that for general graphs $H$ and $\Gamma$, we can lower bound $\mathrm{ex}(\Gamma,\mathcal{P}_H)$ via the extremal numbers of the family of clique quotients of $H$. The corollary is interesting when $H$ is bipartite. For some $H$, we understand $\mathcal{F}_H$ quite well. For example, if $H=C_{2\ell}$ is the even cycle on $2\ell$ vertices, then $\mathcal{F}_H$ consists of all cycles of length at least $\ell$ and at most $2\ell$. It is known \cite{AHL} that if a graph on $n$ vertices has girth larger than $2\ell$ (so the shortest cycle has more than $2\ell$ vertices), then it has at most $\frac{1}{2}n^{1+1/\ell}+\frac{1}{2}n$ edges. It is conjectured that this upper bound is tight up to a constant factor, and this conjecture holds for $\ell \in \{2,3,5\}$, see \cite{FS} for more on this problem. 
We thus conjecture the following: 

\begin{conjecture}\label{gnpcycle}
For every $c > 0$ and integer $\ell$ there exists $C > 1$, such that if $\Gamma$ is a $(c, t)$-sparse graph then $\mathrm{ex}(\Gamma, \mathcal{P}_{C_{2\ell}}) \le C t^{1 - 1/\ell}n^{1 + 1/\ell}$.
\end{conjecture}

For many bipartite $H$, like complete bipartite graphs $K_{s,r}$ with $2 \leq s \leq r$ and $r \geq 3$, every clique quotient of $H$ apart from $H$ is non-bipartite, and hence, for such graphs $H$, the construction used to prove Corollary \ref{lowercor-fam} gives a natural bound of 
$$\mathrm{ex}(\Gamma,\mathcal{P}_{H}) \geq \frac{1}{2}e(\Gamma)\mathrm{ex}(\chi(\overline{\Gamma}),H)/{\chi(\overline{\Gamma}) \choose 2}.$$ It seems challenging to prove a matching upper bound in general.

\medskip
{\bf Induced subdivisions.} Given a graph $H$, let us denote by $\mathcal{D}_H$ the property of not containing a \emph{subdivision} of $H$ as an induced subgraph. In the context of induced Tur\'an-type problem, this family of graphs has received considerable attention recently  \cite{BBCD,DGHMS,GM,KO}. In particular, a result of Hunter, Milojevi\'c, Sudakov, and Tomon \cite{HMST} states that if $G$ has $C t^{O(|V(H)|)}n$ edges and no $K_{t,t}$ as a subgraph (not necessarily induced), then $G$ contains an induced subdivision of $H$. Recasting the problem in our setup, it is natural to believe that the following is true. 

\begin{conjecture} \label{subdivision}
    For every $c > 0$ and a graph $H$, there exists $C > 1$ such that if $\Gamma$ is a $(c, t)$-sparse graph then $\mathrm{ex}(\Gamma, \mathcal{D}_{H}) \le C tn$.
\end{conjecture}

\medskip
{\bf Multicolor induced Ramsey numbers.} For a graph $H$ and positive integer $r$, the induced Ramsey number $r_{\textrm{ind}}(H,r)$ is the minimum $N$ for which there is a graph $G$ on $N$ vertices such that in every $r$-edge-coloring of $G$ there is a monochromatic induced copy of $H$. The existence of induced Ramsey numbers was proved in the early 1970s, and it has remained a challenging open problem to estimate these numbers. As shown in \cite{FS09}, like for Ramsey numbers, if $H$ is non-bipartite, then $$2^r \leq r_{\textrm{ind}}(H,r) \leq r^{C(H)r}$$ for some constant $C=C(H)$. 

One of the consequences of our results implies that for bipartite $H$, the multicolor induced Ramsey number has polynomial growth in the number of colors, similar to Ramsey numbers. Further, we can determine the induced Ramsey number for certain fixed bipartite graphs $H$ up to a lower order factor. 

\begin{theorem}\label{inducedRamsey}
Let $d$ be a positive integer, and fix a $d$-bounded bipartite graph $H$ which contains a $K_{d,s}$ with $s$ sufficiently large. Then $$r_{\textrm{ind}}(H,r)=\tilde{\Theta}(r^d).$$
\end{theorem}

The upper bound follows from our results by considering a stronger density-type variant of induced Ramsey numbers studied by Dudek, Frankl,  and R\"odl \cite{DFR}. Say that a graph $G$ has the \emph{induced $(\varepsilon,H)$ density property} if for every subgraph $G'$ of $G$ with $e(G') \geq \varepsilon e(G)$, there is an induced copy of $H$ in $G$ with all edges in $G'$. If $\varepsilon=1/r$ and $G$ has the induced $(\varepsilon,H)$ density property, then there is a monochromatic induced copy of $H$ in the densest color in any $r$-edge coloring of $G$, thus showing that $r_{\textrm{ind}}(H,r) \leq |G|$. Taking $G=G(n,1/2)$ shows that Theorem \ref{thm:drc} gives the desired upper bound in Theorem \ref{inducedRamsey}. 

The lower bound in Theorem \ref{inducedRamsey} follows from the lower bound on the ordinary Ramsey number for $H$. Using the symmetric hypergraph theorem (see Chapter 4.4 of \cite{GRS}), we can color 
the edges of the complete graph on $n$ vertices with at most $\frac{n^2}{\mathrm{ex}(n,H)}\log n$ colors such that each color class is $H$-free. Together with the lower bound on $\mathrm{ex}(n,K_{d,s})$ for $s$ sufficiently large in $d$, this gives the desired lower bound on the Ramsey number (and hence the induced Ramsey number).


\begin{thebibliography}{10}

\bibitem{AHL} N. Alon, S. Hoory, and N. Linial, The Moore bound for irregular graphs, {\it 
Graphs Combin.} {\bf 18} (2002), 53--57.

\bibitem{AKS23}
N. Alon, M. Krivelevich, and W. Samotij, 
Largest subgraph from a hereditary property in a random graph, 
{\it Discrete Math.} {\bf 346} (2023), Paper No. 113480.  

\bibitem{AKSudakov} 
N. Alon, M. Krivelevich, and B. Sudakov, Tur\'an numbers of bipartite graphs and related Ramsey-type questions, {\it Combin. Probab. Comput.} {\bf 12} (2003), 477--494.


\bibitem{ARS} N. Alon, L. R\'onyai, and T. Szab\'o, Norm-graphs: variations and applications, {\it J. Combin. Theory Ser. B} {\bf 76} (1999), 280--290.

\bibitem{AZ} M. Axenovich and J. Zimmermann, Induced Tur\'an problem in bipartite graphs, arXiv:2401.11296.


\bibitem{bollobas88} B. Bollob\'as, The chromatic number of random graphs, {\it Combinatorica} {\bf 8} (1988), 49-–55.

\bibitem{BBCD} R. Bourneuf, M. Buci\'c, L. Cook and J. Davies, On polynomial degree-boundedness, arXiv:2311.03341. 

\bibitem{Bukh} B. Bukh, Extremal graphs without exponentially small bicliques, arXiv:2107.04167.

\bibitem{DSW}
D. Di Benedetto, J. Solymosi, E. P. White, On the directions determined by a Cartesian product in an affine Galois plane, {\it Combinatorica} {\bf 41} (2021), 755--763.

\bibitem{CLMZ} A. Clifton, H. Liu, L. Mattos and M. Zheng, Subgraphs of random graphs in hereditary families, arXiv:2405.09486.

\bibitem{CFS10}
D. Conlon, J. Fox, and B. Sudakov, An approximate version of Sidorenko's conjecture. {\it Geom. Funct. Anal.} {\bf 20} (2010), 1354-–1366.

\bibitem{DGHMS} X. Du, A. Gir\~ao, Z. Hunter, R. McCarty, and A. Scott, Induced C4-free subgraphs with large average degree, arxiv:2307.08361.



\bibitem{DFR} 
A. Dudek, P. Frankl, and V. R\"odl, Some recent results on Ramsey-type numbers, {\it Discrete Appl. Math.} {\bf 161} (2013), 1197--1202.

\bibitem{ES46} P. Erd\H{o}s and A. M. Stone, On the structure of linear graphs, {\it Bull. Amer. Math. Soc.} {\bf 52} (1946), 1087--1091.

\bibitem{ES66} P. Erd\H{o}s and M. Simonovits, A limit theorem in graph theory, {\it Studia Sci. Math. Hungar.} {\bf 1} (1966), 51--57.

\bibitem{FPS} J. Fox, J. Pach, and A. Suk, Quasiplanar graphs, string graphs, and the Erdős-Gallai problem, Graph drawing and network visualization, 219--231.
Lecture Notes in Comput. Sci., 13764 Springer, Cham, 2023. 

\bibitem{FS09} J. Fox and B. Sudakov, Density theorems for bipartite graphs and related Ramsey-type results, {\it Combinatorica} {\bf 29} (2009), 153--196.

\bibitem{fsdrc} J. Fox and B. Sudakov, Dependent random choice, {\it Random Structures Algorithms} {\bf 38} (2011), 68--99.

\bibitem{Furedi} Z. F\"uredi, On a Tur\'an type problem of Erd\H{o}s, {\it Combinatorica} {\bf 11} (1991),  75--79.  

\bibitem{FS} Z. F\"uredi and M. Simonovits, 
The history of degenerate (bipartite) extremal graph problems, in Erd\H{o}s centennial,  Bolyai Soc. Math. Stud., vol. 25 (J\'anos Bolyai Math. Soc., Budapest, 2013), pp. 169--264.

\bibitem{GM} A. Gir\~ao and Z. Hunter, Induced subdivisions in $K_{s,s}$-free graphs with polynomial average degree, arXiv:2310.18452.  

\bibitem{GRU} M. C. Golumbic, D. Rotem, and J. Urrutia, Comparability graphs and intersection graphs, Discrete Math. 43 (1983), 37--46.

\bibitem{GRS} R.~L. Graham, B.~L. Rothschild, and J.~H. Spencer.
\newblock {\em Ramsey theory}, volume~20.
\newblock John Wiley \& Sons, 1991.

\bibitem{HP} B. Hanson, and G. Petridis, Refined estimates concerning sumsets contained in the roots of unity, {\it Proc. Lond. Math. Soc. (3)} {\bf 122} (2021), 353--358.

\bibitem{HMST} Z. Hunter, A. Milojevi\'c, B. Sudakov and I. Tomon, K\"ovari-S\'os-Tur\'an theorem for hereditary families, arXiv:2401.10853 

\bibitem{KRS} J. Koll\'ar, L. R\'onyai, and T. Szab\'o, Norm-graphs and bipartite Tur\'an numbers, {\it Combinatorica} {\bf 16} (1996), 399--406.

\bibitem{KST} T. K\H{o}v\'ari, V. T. S\'os, and P. Tur\'an. On a problem of K. Zarankiewicz, {\it Colloq. Math.} {\bf 3} (1954), 50--57. 

\bibitem{KO} D. K\"uhn and D. Osthus, Induced subdivisions in $K_{s,s}$-free graphs of large average degree, {\it Combinatorica} {\bf 24} (2004), 287–304.

\bibitem{LTTZ} P.-S. Loh, M. Tait, C. Timmons, and R. M. Zhou, Induced Tur\'an numbers, {\it Combin. Probab. Comput.} {\bf 27} (2018), 274--288.

\bibitem{L} L. Lov\'asz, Perfect graphs, in: Selected Topics in Graph Theory, vol. 2, Academic Press, London, 1983, 55--87.

\bibitem{MST1} A. Milojevi\'c, B. Sudakov, and I. Tomon, Incidence bounds via extremal graph theory, arXiv:2401.06670

\bibitem{MST2} A. Milojevi\'c, B. Sudakov, and I. Tomon, Point-variety incidences, unit distances and Zarankiewicz’s problem for algebraic graphs, arXiv:2403.08756

\bibitem{PT} J. Pach and G. T\'oth, Comment on Fox News, {\it Geombinatorics} {\bf 15} (2006), 150--154.

\bibitem{V} S. P. Vadhan, Pseudorandomness, {\it Foundations and Trends in Theoretical Computer Science} {\bf 7} (2012), No. 1–3, 1--336.

\bibitem{Yip} C. H. Yip, On the clique number of Paley graphs and generalized Paley graphs, MSc thesis, University of British Columbia (2021).

\bibitem{Zim} J. Zimmermann, Induced Tur\'an problems, BSc thesis, Karlsruhe Institute of Technology (2024).

\end{thebibliography}
\end{document}